\documentclass[12pt,a4paper]{article}

\usepackage{amsmath,amsthm}
\usepackage{amsfonts}
\usepackage{amssymb}
\usepackage{calligra}
\usepackage{calrsfs}
\usepackage{hyperref}
\usepackage{enumitem}
\usepackage[nameinlink,noabbrev,capitalise]{cleveref}
\usepackage{float}
\usepackage{afterpage}

\usepackage[left=2cm,right=2cm,top=2cm,bottom=2cm]{geometry}
\usepackage[mathscr]{euscript}
\usepackage{tikz-cd}
\usetikzlibrary{shapes,arrows}

\DeclareMathAlphabet{\mathcalligra}{T1}{calligra}{m}{n}

\theoremstyle{plain} 
\newtheorem{theorem}{\indent\sc Theorem}[section]
\newtheorem{lemma}[theorem]{\indent\sc Lemma}
\newtheorem{corollary}[theorem]{\indent\sc Corollary}
\newtheorem{proposition}[theorem]{\indent\sc Proposition}
\newtheorem{remark}[theorem]{\indent\sc Remark}
\newtheorem{definition}[theorem]{\indent\sc Definition}
\newtheorem{example}[theorem]{\indent\sc Example}

\numberwithin{equation}{section}

\newcommand{\Der}{\mathop{\mathrm{Der}}}

\newcommand{\DS}[1]{\mathop{\mathrm{DS}({#1})}}
\newcommand{\EP}[1]{\mathop{\mathrm{exp}({#1})}}
\newcommand{\PO}[1]{\mathop{\mathrm{POexp}({#1})}}
\newcommand{\Sym}{\mathop{Sym}}
\newcommand{\K}{\mathbb{K}}
\newcommand{\A}{\mathcal{A}}
\newcommand{\B}{\mathcal{B}}
\newcommand{\Z}{\mathbb{Z}}
\newcommand{\C}{\mathcal{C}}
\crefname{enumi}{Case}{Cases}

\title{Free resolution of the logarithmic derivation modules 
of close to
free arrangements}

\author{Junyan CHU\thanks{Graduate School of Mathematics, Kyushu University, Fukuoka, 819-0395, Japan. \\ 
 Email:chu.junyan.949@s.kyushu-u.ac.jp. 
The author has been supported by the China Scholarship Council.
}}

\begin{document}

\maketitle

\begin{abstract}

This paper studies the algebraic structure of a new class of hyperplane arrangement $\A$ obtained by deleting two hyperplanes from a free arrangement.
We provide information on the minimal free resolutions of 
the logarithmic derivation module of $\A$,
 which can be used to compute a lower bound for the graded Betti numbers of the resolution.

Specifically, for the three-dimensional case, 
we determine 
the minimal free resolution of 
the logarithmic derivation module of $\A$.
We present illustrative examples of our main theorems to provide insights into the relationship between algebraic and combinatorial properties for close-to-free arrangements.

\end{abstract}

\section{Introduction}

Let $V$ be the $\ell$-dimensional vector space ${\K}^{\ell}$ over a field $\K$. The coordinate ring $S = \Sym(V^*) \cong \K [x_1, \ldots, x_{\ell}]$ is equipped with the usual grading and its degree $i$ homogeneous part is denoted by $S_i$. 
A (central) \emph{hyperplane arrangement} $\A$ is a finite set of linear hyperplanes in $V$. 
For a hyperplane $H\in \A$, the defining  linear form is denoted by $\alpha_H\in S_1$ with $H=\ker{\alpha_H}$. 
The \emph{defining polynomial} $Q(\A)$ of $\A$ is defined as $Q(\A) = \prod_{H \in \A} \alpha_H$, and it is defined up to a scalar multiple.
One of the most important algebraic invariants
associated to an arrangement $\A$ is its \emph{logarithmic derivation module} $D(\A)$ defined by
\begin{align*}
  D(\A)=\{\theta \in \Der S \mid  \theta(\alpha_H)\in S\alpha_H\  \text{for any} \  H=\ker \alpha_H\in \A \}, 
\end{align*}
where $\Der{S}$ is the free $S$-module 
of derivations generated by $\{\partial_{x_i}\mid 1\le i\le \ell\}$.
Given a derivation $\theta=\sum_{i=1}^{\ell}f_i\partial_{x_i}\in D(\A)$, we say it is homogeneous if  all $f_i\in S_d$ for some $d\in {\Z}_{\geq 0}$, and we write $\deg \theta=d$. 
Generally, $D(\A)$ is a reflexive graded $S$-module \cite{ref_17_in_spog} and is not always free. 
If $D(\A)$ is free, then there exist homogeneous derivations $\theta_1, \ldots, \theta_{\ell} \in D(\A)$ such that $D(\A) = \bigoplus_{i=1}^{\ell} S\theta_i$. In this situation, we say that $\A$ is free with \emph{exponents} $\EP{\A} = (d_1, \ldots, d_{\ell})$, where $d_i = \deg \theta_i$.

The study of $D(\A)$ has focused primarily on the case when it is free (see \cite{yoshinaga2014} for a survey). Much remains unexplored when it is not free.
In order to understand non-free cases, a natural approach is to look at their graded minimal free resolution.
Some works consider the degrees (the Betti numbers) of the graded minimal free resolution 
\cite{Derksen2004,graded_betti2010,saito2019degeneration}.
In particular, \cite{11_in_degree_seq,dgree_seq} study the \emph{derivation degree sequence}, denoted as $\DS{\A}$, which is defined as the unordered sequence of the degrees of the minimal homogeneous generators of $D(\A)$.
Moreover, we use $|\DS{\A}|$ to represent the number of the minimal homogeneous generators of $D(\A)$. 
In cases where $\A$ is free, we have $\DS{\A} = \EP{\A}$.
Although minimal homogeneous generators are not unique, 
their degrees do not depend on the choice since they are the degrees of $\mathrm{Tor}_0(\K, D(\A))$. 
One of the main challenges in the study of non-free arrangements is that the determination of these two algebraic properties of  $D(\A)$ 
is generally influenced not only by the combinatorics of  $\A$ but also by its geometry (see \cref{ex:DS_geom}). 

To tackle this issue, our initial approach involves examining arrangements that are close to free arrangements.
This is inspired by the \emph{next-to-free minus (NT-free-1)} defined by Abe.
\begin{definition}[Definition 1.3 and 6.1 in \cite{spog}]
    We say that $\mathcal{B}$ is next-to-free minus (NT-free-1) if there exist a free arrangement $\mathcal{A}$ and a hyperplane $H \in \mathcal{A}$ such that $\mathcal{B} = \mathcal{A} \setminus {\{H\}}$.

    In this case, we say $\A$ is a \textbf{free addition} of $\B$.
\end{definition}

Now we introduce a class of arrangement,  which has a nice structure called \emph{strictly plus-one generated (SPOG)}, closely related to the NT-free-1 arrangement.

\begin{definition}[Definition 1.1 in \cite{spog}]\label{spogarr}
An arrangement $\B$ is said to be 
    \emph{strictly plus-one generated (SPOG)} with exponents $\PO{\B}=(d_1,\ldots,d_{\ell})$ and level $d$, if there exist $f_1,\ldots,f_{\ell},\alpha\in S$ with $\alpha\neq 0$ such that $D(\B)$ has a minimal free resolution of the following form:
    \begin{align*}
        0 \rightarrow S[-d-1] \xrightarrow{(\alpha,f_1,\ldots,f_{\ell})} S[-d] \oplus \left(\bigoplus_{i=1}^{\ell} S[-d_i]\right) \rightarrow D(\B) \rightarrow 0.
    \end{align*}
   In particular, $\DS{\B}=(d_1,\ldots,d_{\ell},d)=(\PO{\B},d)$.
\end{definition}
{
\begin{remark}\label{rmk-spog}
 In other words, $\B$ is SPOG if there is a minimal set of homogeneous generators $ \theta_1, \theta_2, \ldots, \theta_{\ell}$ and $\varphi$ for $D(\B)$ such that $\deg \theta_i = d_i$, $\deg \varphi = d$, and
    \[
    \sum_{i=1}^{\ell} f_i \theta_i + \alpha \varphi = 0,
    \]
    where $f_i \in S$ and $0\neq \alpha \in S_1$. 
    This $\alpha$ is called the level coefficient, and $\varphi$ is a level element. 
Moreover, when $d = d_i$ for some $i$ and $f_i\neq 0$, then $\theta_i$ also can be a level element. In other words, the choice of the level coefficient and element is not unique. 
\end{remark}}

Abe \cite{spog} shows that an NT-free-1 arrangement $\B$ is either free or SPOG.
If $\B$ is SPOG and NT-free-1, the level of $\B$ can be determined by the combinatorial properties of its free addition. 
To state his results, let us recall that the definition of the intersection lattice of $\A$, denoted by $L(\A)$, as follows:
\begin{align*}
    L(\A):=\left\{\bigcap_{H\in \B } H\mid  \B \subset\A\right\},
\end{align*}
where $L(\A)$ is equipped with a partial order induced by reverse inclusion. 
For a given $X\in L(\A)$, the localization $\A_X$ of $\A$ at $X$ is defined by
\begin{align*}
    \A_X:=\{H\in \A\mid  X\subset H\},
\end{align*}
and the restriction $\A^X$ of $\A$ onto $X$ is defined by
\begin{align*}
    \A^X:=\{H\cap X\mid  H\in \A\setminus \A_X\}.
\end{align*}

The following is a significant theorem about NT-free-1 arrangements, which plays a crucial role in shaping our results.
  
  \begin{theorem}[Theorem 1.4 and Proposition 5.3 in  \cite{spog}]\label{diml}
Let $\A$ be free with $\EP{\A}=(d_1,\ldots, d_{\ell})$ and  $H\in \A.$ Then $\A^{\prime}=\A\setminus{\{H\}}$ is free, or SPOG with $\PO{\A^{\prime}}=(d_1,\ldots, d_{\ell})$ and level $d=|\A^{\prime}| -|\A^H|$.  Moreover, if $\ell=3$, then $d\geq \max \{d_1,d_2,d_3\}$.
\end{theorem}
{
\begin{remark}\label{rmk-levelEl}
Suppose that $\A$ is free and $\A' = \A \setminus H$ is SPOG.
   \begin{enumerate}[label=(\arabic*)]
   \item \label{rmk1.3-1}   

Assume that \(\theta_1,\ldots,\theta_\ell\) is a basis of \(D(\A)\). Since \(\DS{\A'}=(\PO{\A},d)\) and \(\PO{\A'}=\EP{\A}\), there exists a \(\varphi \in D(\A')\) of degree \(d\) such that \(\theta_1,\ldots,\theta_\ell,\varphi\) generate \(D(\A')\). Since \(D(\A)\subsetneq D(\A')\), we have \(D(\A') = D(\A) + S\varphi\), and \(\varphi \notin D(\A)\) as a level element of \(D(\A')\).

       \item 

        Let $\varphi \notin D(\A)$ be a level element of $\A'$. Let $\alpha_H \in S_1$ be such that $\ker \alpha_H = H$. Since $\varphi \in D(\A') \setminus D(\A)$, we have $\alpha_H \varphi \in D(\A)$. Thus, we obtain a relationship between the minimal generators of $D(\A')$. Since SPOG arrangements have a unique relation among the minimal generators, we may always assume that $\alpha_H$ is the level coefficient.

   \end{enumerate}
   
\end{remark}}

In this paper, we introduce a new class of possibly non-free arrangements
obtained by deleting two hyperplanes from free arrangements.
\begin{definition}
    We say that $\mathcal{B}$ is \emph{next-to-free-minus-two} (NT-free-2) if there exist a free arrangement $\A$ and two hyperplanes $H_1, H_2 \in \A$ such that $\mathcal{B} = \A \setminus {\{H_1, H_2\}}$.
\end{definition}

By analyzing the minimal free resolution, we deduce the following theorem:
\begin{theorem}\label{proj_dim_Aij}
Let $\B$ be a NT-free-2 arrangement. Then the projective dimension $\text{pd}_S(D(\B))\leq 1$ if and only if  $|\DS{\B}|\leq \ell +2$.
\end{theorem}

Let $\A=\{H_1,\ldots,H_p\mid H_i=\ker\alpha_i\}$ be a free arrangement.
We denote the NT-free-1 arrangement  $\A \setminus \{H_j\}$ by $\A_j$
and the NT-free-2 arrangement $\A \setminus \{H_j, H_k\}$ by $\A_{j,k}$.
We note that $\A_{j,k}=\A_{k,j}$.
We also denote $\A^{H_j}$ by $\A^j$.

If $\A_1$ or $\A_2$ is free, 
then $\A_{1,2}$ is NT-free-1 and 
it is either free or SPOG by \cref{diml}.
Therefore, we focus on the case when 
none of $\A_1$ and $\A_2$ are free.
We show that when $|\DS{\A_{1,2}}|\leq \ell + 2$, the minimal free resolution of $D(\mathcal{A}_{1,2})$ assumes one of the forms listed in \cref{free_resolution_Aij}. 
When $|\DS{\A_{1,2}}|> \ell +2$, we obtain a lower bound for the Betti numbers of $D(\A_{1,2})$ using the information provided in \cref{free_resolution_>l+2_M0}. 

Note that since 
the logarithmic derivation module
is a reflexive graded $S$-module, 
its projective dimension is less than or equal to $\ell - 2$
(see, \cref{pdim<l-2}).
Consequently, for a three-dimensional NT-free-2 arrangement $\A_{1,2}$, we can infer that $\text{pd}_S(D(\A_{1,2})) \leq 3 - 2 = 1$. Utilizing \cref{proj_dim_Aij}, we then establish that $|\DS{\A_{1,2}}| \leq \ell + 2=5$.
Furthermore, we determine the precise form of the minimal free resolution of $D(\A_{1,2})$.
\begin{theorem}\label{free_resolution_in_dim3}
Assume $\ell=3$. Let $\A$ be free with $\exp(\A)=(d_1,d_2,d_3)$, and $\A_1$ and $\A_2$ be SPOG with levels $c_1$ and $c_2$, respectively. We may assume $c_1\leq c_2$.{ Then $\A_{1,2}$ can never be free in this set-up, } and
the module $D(\A_{1,2})$  has a minimal free resolution in one of the following forms:
\begin{enumerate}[label=(\arabic*)]
\item Suppose that $|\A_{H_1\cap H_2}|=2$.

There exist $f_i,g_i\in S$ such that
    \begin{align*}
        0 \rightarrow S[-c_1-1] \oplus S[-c_2-1] &\xrightarrow{\begin{array}{c} (\alpha_1,0, f_1,f_2,f_3) \\ (0,\alpha_2, g_1,g_2,g_3) \end{array}}\\ & S[-c_1] \oplus S[-c_2] \oplus \left(\bigoplus_{i=1}^{3} S[-d_i]\right) \rightarrow D(\A_{1,2}) \rightarrow 0,
    \end{align*}
where $i=1,2,3$.
\item Suppose that $|\A_{H_1\cap H_2}|>2$.
\begin{enumerate}[label=(\arabic{enumi}.\arabic*)]

    \item \label{useless-thm1p3} If $c_1=c_2$,
    there exist $f_i\in S$  such that
     \begin{align*} 
        0 \rightarrow S[-c_1-1] \xrightarrow{(\alpha_1\alpha_2,f_1,f_2,f_3)} S[-c_1+1] \oplus \left(\bigoplus_{i=1}^{3} S[-d_i]\right) \rightarrow D(\A_{1,2}) \rightarrow 0.
    \end{align*}
    Moreover, $\A_{1,2}$ is SPOG if and only if $c_1=c_2=\max\{d_1,d_2,d_3\}$. {If $\A_{1,2}$ is SPOG, let $d_2=\max\{d_1,d_2,d_3\}$, then $f_2 \in S_1$.}
        
    \item If $c_1<c_2$, there exist $f_i,g_i,g\in S$ with $g\neq 0$ such that
    \begin{align*}
        0 \rightarrow S[-c_1-1] \oplus S[-c_2]  & \xrightarrow{\begin{array}{c} (\alpha_1,0, f_1,f_2,f_{3}) \\ (g,\alpha_2, g_1,g_2,g_{3}) \end{array}} \\ & S[-c_1] \oplus S[-c_2+1] \oplus \left(\bigoplus_{i=1}^{3} S[-d_i]\right)
        \rightarrow D(\A_{1,2}) \rightarrow 0,
    \end{align*}
    where $i=1,2,3$.

    \end{enumerate}
\end{enumerate}
\end{theorem}

The organization of this article is as follows: 
In Section 2, we introduce previous results and definitions. 
In Sections 3 and 4, we focus on proving the main results. Specifically, in Section 3, we present the proofs and provide examples in a general-dimensional setting. In Section 4, we shift our attention to the three-dimensional case and provide further examples.

\vspace{5mm}
\emph{Acknowledgements:}
We would like to thank Takuro Abe, Shizuo Kaji, Paul Muecksch, and Akiko Yazawa for many helpful discussions. We are particularly grateful to Takuro Abe for suggesting this problem and for using multiarrangemet to prove the three-dimensional case. 
The author has been supported by the China Scholarship Council.


\section{ Preliminaries}

In this section, we will summarize several results and definitions. 
Let $\A$ be an arrangement and $L(\A)$ be its intersection lattice.

Let $\theta_E=\sum_{i=1}^{\ell}x_i\partial_{x_i}$ be the \emph{Euler derivation}, which is a homogeneous derivation of $\deg \theta_E=1$ and always contained in $D(\A)$. Recall that for  every $H\in \A$, there exists a decomposition:
\begin{align}\label{eq:splitting}
    D(\A)\cong S\theta_E\oplus D_H(\A),
\end{align}
where $D_H(\A):=\{\theta\in D(\A)\mid \theta(\alpha_H)=0\}$ (see, for example,  \cite[Lemma 1.33]{O-T}).
This implies that $1\in \DS{\A}$. Furthermore, if $\A\neq \emptyset$ is free with $\EP{\A}=(d_1,\ldots,d_{\ell})$, we may assume that $d_1=\deg \theta_E=1$.

The following has been well-known and frequently utilized by specialists.
\begin{proposition}[\cite{11_in_degree_seq}]
  Let $H\in \A$. Then there is a polynomial $B$ of degree $|\A|-|\A^H|-1$ such that $\alpha_H \nmid B$, and
  $\theta(\alpha_H)\in (\alpha_H, B)$
for all $\theta\in D(\A\setminus \{H\})$. 
\end{proposition}

From the above proposition, we can easily derive the following proposition. We include it here since we frequently utilize it.

\begin{proposition}[Corollary 3.3 in  \cite{spog}]\label{sum} 
  Let $H\in \A$ and $\A'=\A\setminus \{H\}$. Assume that there exists $\varphi\in D(\A')$ with $\deg \varphi=|\A'| -|\A^H|$ such that $\varphi\notin D(\A)$. Then $D(\A')=D(\A)+S\varphi$.
\end{proposition}

We now define the \emph{Euler restriction map}  $\rho: D(\A)\rightarrow D(\A^H)$ for an arrangement $\A$ by taking modulo $\alpha_H$. Additionally, let $\A'=\A\setminus \{H\}$.
    We have an exact sequence as follows:
{\begin{proposition}[Proposition 4.45 in \cite{O-T}]\label{exact_seq}

    \begin{align*}
        0\rightarrow D(\A')\xrightarrow{\cdot \alpha_H}D(\A)\xrightarrow{\rho}D(\A^H).
    \end{align*}
\end{proposition}}

Moving forward, let's delve into some results pertaining to the logarithmic derivation module $D(\A)$ and its freeness.


\begin{theorem}[{Addition-Deletion Theorem.} Removal Theorem in \cite{11_in_degree_seq}]\label{a-d-thm}
Let $H\in\A$, $\A':=\A \setminus \{H\}$ and $\A^{''}:=\A^H$. Then two of the following imply the third:

\begin{enumerate}[label=(\arabic*)]
 \item $\A$ is free with $\EP{\A}=(d_1,\ldots,d_{\ell})$.
 \item  $\A'$ is free with $\EP{\A'}=(d_1,\ldots,d_{\ell-1},d_{\ell}-1)$.
 \item  $\A''$ is free with $\EP{\A''}=(d_1,\ldots,d_{\ell-1})$.
\end{enumerate}
Moreover, all the three above hold true  if $\A$ and $\A'$ are free.
\end{theorem}

\begin{theorem}[Saito's criterion. \cite{ref_17_in_spog}]\label{saito}
    Let $\theta_1,\ldots,\theta_{\ell}\in D(\A)$ be homogenous and linearly independent over $S$. Then $\A$ is free with basis $\{\theta_1,\ldots,\theta_{\ell}\}$  if and only if $$\sum_{i=1}^{\ell}\deg \theta_i=|\A| .$$
\end{theorem}

Let us introduce the multiarrangement theory.
A \emph{multiarrangement} is a pair $(\A,m)$,  where $\A$ is a hyperplane arrangement in $V$ and multiplicity $m$ is a map: $\A\rightarrow \Z_{\geq 0}$. Define $|m|=\sum_{H\in \A}m(H)$.
If $m(H)=1$ for all $H\in \A$, we say that the multiarrangement $(\A,m)$ is a hyperplane arrangement, which is also called a simple arrangement.
The \emph{logarithmic derivation module} $D(\A,m)$ is defined as follows:
\begin{align*}
  D(\A,m)=\{\theta \in \Der S\mid \theta(\alpha_H)\in S\alpha_H^{m(H)}\  \text{for   any} \  H=\ker \alpha_H\in \A \}.  
\end{align*}
The module $D(\A,m)$ is also a reflexive graded $S$-module, which is not always free. We can define the concepts of freeness and exponents for $(\A,m)$ in the same way as for simple arrangements.


\begin{definition}[\cite{ref_24_in_spog}] 
For an arrangement $\A$  and $H\in \A$, define the Ziegler multiplicity $m^H: \A^H\rightarrow \Z_{\geq 0}$ by $m^H(X):=| \{L\in \A\setminus\{H\}\mid  L\cap H=X\}| $ for $X\in \A^H$. The pair $(\A^H,m^H)$ is called the \emph{Ziegler restriction} of $\A$ onto $H$. Also, there is a \emph{Ziegler restriction map}:
\begin{align*}
    \pi: D_H(\A)\rightarrow D(\A^H,m^H)
\end{align*}
by taking modulo $\alpha_H$. In particular, there is an exact sequence:
\begin{align*}
    0\rightarrow D_H(\A)\xrightarrow{\cdot \alpha_H} D_H(\A)\xrightarrow{\pi} D(\A^H,m^H).
\end{align*}
\end{definition}

\begin{theorem}[Theorem 11 in \cite{ref_24_in_spog}]\label{Ziegler free}
    Assume that $\A$ is free with  exponents $\EP{\A}=(1,d_2,\ldots,d_{\ell})$. Then for any $H\in \A$, the Ziegler restriction $(\A^H,m^H)$ is also free with $\EP{\A^H,m^H}=(d_2,\ldots,d_{\ell})$. {Explicitly}, for the Ziegler restriction $\pi: D_H(\A)\rightarrow D(\A^H,m^H)$, any basis $\theta_2,\ldots,\theta_{\ell}$ for $D_H(\A)$  such that $\pi(\theta_2),\ldots,\pi (\theta_{\ell})$ {form} a basis for $D(\A^H,m^H)$. In particular, $\pi$ is surjective when $\A$ is free.
\end{theorem}

\begin{lemma}[Lemma 4.2 and Lemma 4.3 in \cite{abe-n}]\label{multi -1}
  Let $\A$ be central line arrangement and let $m,m'$ be multiplicities on $\A$ such that $|m|=|  m'| +1$ and $m(H)\geq m'(H)$ for any $H\in \A$. 
  If $\EP{\A,m'}=(a,b)$, then $\EP{\A,m}=(a+1,b)$ or $(a,b+1)$.
\end{lemma}

Finally, we present the basics of free resolutions.

\begin{definition}\label{def_resol}
    For arrangement $\B$, we denote the minimal free resolution of the module $D(\B)$ by 
\begin{align*}
    0 \rightarrow M_k \xrightarrow{R_k} M_{k-1} \xrightarrow{R_{k-1}} \cdots 
    \xrightarrow{R_{2}} M_{1} \xrightarrow{R_{1}} M_{0} \rightarrow  
    D(\B) \rightarrow 0,
\end{align*}
where $k=pd_S(D(\B))$ is the projective dimension of $D(\B)$.
Here, $R_i \ (i=1,\ldots,k)$ are represented by matrices acting by multiplication from the left.
We use $R_i(j)$ to denote the $j$-th row of $R_i$, and $R_i(j_1,j_2)$ to denote the $(j_1,j_2)$ entry of $R_i$.
\end{definition}

Let $JQ(\A)$ be the Jacobian ideal of $Q(\A)$
generated by $\frac{\partial Q(\A)}{\partial x_i}$
for $i=1,\ldots,n$.
Since 
$D(\A)=\{\theta \in \Der S \mid  \theta(Q(\A))\in SQ(\A)\}$,
we have a free resolution of the form
\begin{equation}\label{eq:syz}
0\to N_\ell\to \cdots \to N_3 \to D_H(\A)\to S^n \to S \to S/JQ(\A)\to 0
\end{equation}
by Hilbert's syzygy theorem.
Since $pd_S(D(\A))=pd_S(D_H(\A))$
by \eqref{eq:splitting}, we have
\begin{lemma}\label{pdim<l-2}
    $pd_S(D(\A))\le \ell -2$.
\end{lemma}
From \eqref{eq:syz}, 
we see $D(\A)$ is reflexive since it is a second syzygy.


The following theorem is a simplified version of the result found in the reference \cite{graded_betti2010}.

\begin{theorem}[Theorem 0.2 in \cite{graded_betti2010}]\label{betti-1}
    If the logarithmic derivation module $D(\B)$ has a free resolution given by:
    \begin{align*}
        0 \rightarrow \bigoplus_{i=1}^{r_k} S[-d_i^k] \rightarrow \cdots 
        \rightarrow \bigoplus_{i=1}^{r_1} S[-d_i^1] \rightarrow \bigoplus_{i=1}^{r_0} S[-d_i^0] \rightarrow D(\B) \rightarrow 0,
    \end{align*}
    then $|\B| = \sum_{j=0}^k (-1)^j \sum_{i=1}^{r_j} d_i^j$.      
\end{theorem}

As a corollary,
we have the following property
for SPOG arrangements:
\begin{proposition}[Proposition 4.1 in \cite{spog}\label{coeff b_2}]
    Let $\B$ be SPOG with $ \PO{\B}=(d_1,d_2,\ldots,d_{\ell})$ and level $d$.  Then $\sum_{i=1}^{\ell}d_i-1=|\B| $.
\end{proposition}


\section{The Minimal Free Resolution of NT-Free-2 Arrangements}

In this section, we introduce some notations. 
If $U$ is a subset of $\Der S$, then $SU:=\sum_{\xi\in U}S\xi$. 
Throughout this section, we assume that $\A=\{H_i\mid H_i:\alpha_i=0\}$ is free, with $\EP{\A}=(1, d_2, \ldots, d_{\ell})$, and $\A_j$ is SPOG with $\PO{\A}=(1, d_2, \ldots, d_{\ell})$ and level $c_j$ for $j=1,2$.  Let $c_1\leq c_2$. 
By \cref{exact_seq}, we have \cref{fig:diafram}.

\begin{figure}[H]
\begin{center}
\begin{tikzpicture}
  \matrix (mat) [matrix of nodes, row sep=1cm, column sep=1cm] {
    & $0$ & $0$ &  \\
    $0$ & $D(\A_{1,2})$ & $D(\A_2)$ & $D((\A_2)^1)$  \\
    $0$ & $D(\A_1)$ & $D(\A)$ & $D(\A^1)$ \\
    & $D((\A_1)^2)$ & $D(\A^2)$ & 
    \\
  };

  \draw[->] (mat-2-1) --  (mat-2-2);
  \draw[->] (mat-2-2) -- node[above] {$\cdot \alpha_1$} (mat-2-3);
  \draw[->] (mat-2-3) -- node[above] {$\rho_{2}^1$} (mat-2-4);

 \draw[->] (mat-3-1) -- (mat-3-2);
  \draw[->] (mat-3-2) -- node[above] {$\cdot \alpha_1$} (mat-3-3);
  \draw[->] (mat-3-3) -- node[above] {$\rho^1$} (mat-3-4);


   \draw[->] (mat-1-2) -- (mat-2-2);
  \draw[->] (mat-2-2) -- node[right] {$\cdot \alpha_2$} (mat-3-2);
  \draw[->] (mat-3-2) -- node[right] {$\rho_{1}^2$} (mat-4-2);

  \draw[->] (mat-1-3) -- (mat-2-3);
  \draw[->] (mat-2-3) -- node[right] {$\cdot \alpha_2$} (mat-3-3);
  \draw[->] (mat-3-3) -- node[right] {$\rho^2$} (mat-4-3);

\label{pic-module}
\end{tikzpicture}
\end{center}
 \caption{Exact Sequence Diagram}
\label{fig:diafram}
\end{figure}

\begin{remark}
\ \par
\begin{enumerate}[label=(\arabic*)]
\item {By \cref{rmk-levelEl} \ref{rmk1.3-1}, we may assume that the level element for $\A_j$ is not in $D(\A)$.}

\item As there is no confusion, we can simplify by stating that $D(\A_2^1):=D((\A_2)^1)$ and $D(\A_1^2):=D((\A_1)^2)$.

\end{enumerate}
\end{remark}

\begin{lemma}\label{cannot}
Let $\A_i$ be SPOG, where $H_i\in \A$. 
For any  homogeneous  basis element
$\theta$ for $D(\A)$, we have 
$\theta\notin \ker \rho^i$.
\end{lemma}

\begin{proof}
Assume that 
$\theta_1,\ldots,\theta_{\ell}$ is a basis for $D(\A)$ such that $\theta_1\in \ker \rho_i$.
By \cref{exact_seq}, there is an element $\varphi\in D(\A_i)$ such that $\theta_1=\alpha_i\varphi$.
Since $\varphi,\theta_2,\ldots,\theta_\ell \in D(\A_i)$ are $S$-independent, 
by  \cref{saito}, we may get that $\A_i$ is free, which is a contradiction. 
\end{proof}

\begin{lemma} \label{level_elt_with_ker} 
We have $D(\A)+\ker \rho_i^j\subset D(\A_i)$, where $i\neq j\in \{1,2\}$. Moreover,
   \begin{enumerate}[label=(\arabic*)]
\item \label{useless_1} there exists an element
  $\varphi\in D(\A_{1,2})$ such that 
$\alpha_2\varphi$ is the level element for $\A_1$ if and only if 
$D(\A_1)=D(\A)+\ker \rho_1^2$. 

\item     \label{useless_1'} 
   there exists an element 
  $\varphi\in D(\A_{1,2})$ such that 
$\alpha_1\varphi$ is the level element for $\A_2$ if and only if 
$D(\A_2)=D(\A)+\ker \rho_{2}^1$.
\end{enumerate}
\end{lemma}

\begin{proof}
By \cref{fig:diafram}, it follows that
 $\ker \rho_i^j\subset D(\A_i)$. 
 Note that $D(\A)\subset D(\A_i)$; hence, we have $D(\A)+\ker \rho_i^j\subset D(\A_i)$.
     \begin{enumerate}[label=(\arabic*)]
\item 
By  \cref{exact_seq}, we can deduce the following equivalences:
\begin{align*}
 & \text{There exists an element }
  \varphi\in D(\A_{1,2}) \text{ such that }
\alpha_2\varphi \text{ serves as the level element for } \A_1.\\
\iff & \text{There exists a level element }
  \theta\in D(\A_{1}) \text{ such that }
\theta\in \ker \rho_1^{2}.
\end{align*}
By \cref{rmk-levelEl}, we can represent $D(\A_1)$ as $D(\A) + S\theta$. Consequently, we can state the following equivalences:
\begin{align*}
& \text{There exists a level element }
  \theta\in D(\A_{1}) \text{ such that }
\theta\in \ker \rho_1^{2}.\\
\iff & D(\A_1) = D(\A) + \ker \rho_1^2.
\end{align*}

 \item  This scenario bears resemblance to the one discussed in \cref{useless_1}.
 \end{enumerate}

\end{proof}

\begin{lemma}\label{elt_inA_belong_to_B}
For every element $\varphi \in D(\A_{1,2})$, if $\alpha_j\varphi \in D(\A)$ for some $j \in \{1,2\}$, then $\varphi \in D(\A_j)$. Moreover, if both $\alpha_1\varphi$ and $\alpha_2\varphi$ are in $D(\A)$, then $\varphi$ itself is an element of $D(\A)$.
\end{lemma}

\begin{proof}
    Let $\varphi \in D(\A_{1,2})$ be given, and assume that  $\alpha_1\varphi \in D(\A)$. 
    By \cref{exact_seq,fig:diafram}, we have $\alpha_1\varphi \in \ker \rho^1= \alpha_1 D(\A_1)$.
    Hence,  we can conclude that $\varphi \in D(\A_1)$.
    
    Similarly, if $\alpha_2\varphi \in D(\A)$, we can deduce that $\varphi \in D(\A_2)$.
    
    Moreover, if $\alpha_1\varphi \in D(\A)$ and $\alpha_2\varphi \in D(\A)$, then $\varphi \in D(\A_1)\cap D(\A_2)= D(\A)$. 
\end{proof}

\begin{lemma}\label{inters=2}  
    If $|\A_{H_1\cap H_2}| =2$, then $D(\A_{1,2})=D(\A_1)+D(\A_2)$.
  \end{lemma}

  \begin{proof}
Let $\theta_{{\ell},2}$ be a  level element for $D(\A_2)$.
Since $|\A_{H_1\cap H_2}| =2$, we can deduce that $|\A^{2}| = |\A_1^{2}| + 1$. Consequently, we have $\deg \theta_{\ell ,2} = |\A_2| - |\A^{2}| = (1 + |\A_{1,2}|) - (|\A_1^{2}| + 1) = |\A_{1,2}| - |\A_1^{2}|$.
 Importantly, $\theta_{{\ell},2} \in D(\A_{1,2}) \setminus D(\A_1)$.{ According to \cref{sum}, this implies that $D(\A_{1,2}) = D(\A_1) + S\theta_{{\ell},2}$. Since $\theta_{{\ell},2}\in D(\A_2)$, we have $D(\A_{1,2})=D(\A_1)+D(\A_2)$.}
  \end{proof}

  
 \begin{lemma} \label{cond_for_inter>2} 
    If $D(\A)+\ker \rho_2^1= D(\A_2)$, then $|\A_{H_1\cap H_2}|>2$.
\end{lemma}

\begin{proof}
By \cref{level_elt_with_ker} \ref{useless_1'}, we may assume that $\alpha_1\varphi$ is a level element for $D(\A_2)$. If $|\A_{H_1\cap H_2}| = 2$, it follows that $D(\A_{1,2})=D(\A_1)+D(\A_2)$ by \cref{inters=2}. This observation further implies that 
$$\varphi\in D(\A_{1,2})_{<c_2}\subset D(\A_1)_{<c_2}+D(\A_2)_{<c_2}\subset D(\A_1)+D(\A)\subset D(\A_1).$$
Thus, $\alpha_1\varphi \in D(\A)$, which is a contradiction with $\alpha_1\varphi \notin D(\A)$ being a level element for $D(\A_2)$. Hence, $|\A_{H_1\cap H_2}|>2$.
\end{proof}

\begin{proposition}\label{min_gen_set}
Let $\{\theta_1, \ldots, \theta_\ell\}$ be a basis for $D(\A)$.
There exist two level elements, $\theta_{{\ell},1}$ for $D(\A_1)$ and $\theta_{{\ell},2}$ for $D(\A_2)$, such that a minimal generator set for $D(\A_{1,2})$ falls into one of the following cases:

  \begin{enumerate}[label=(\arabic*)]

  \item \label{useless-2}   If  
  $D(\A)+\ker \rho_1^2\subsetneq D(\A_1)$ and $D(\A)+\ker \rho_2^1= D(\A_2)$, then  $| \A_{H_1\cap H_2}|>2$, and a minimal generator set for $D(\A_{1,2})$ can be expressed as either $\{\theta_1, \ldots, \theta_\ell,\varphi_2, \theta_{{\ell},1} \mid \theta_{{\ell},2}=\alpha_1\varphi_2\}$ or $\{\theta_1, \ldots, \theta_{\ell-1}, \varphi_2, \theta_{{\ell},1} \mid c_1 < c_2 = \deg \theta_{\ell}, \alpha_2\theta_{{\ell},2}=\alpha_2\alpha_1\varphi_2\in  S\theta_1 + \cdots + S\theta_{\ell-1}+\K^*\alpha_1\theta_\ell\}$, where $\K^*=\K\setminus \{0\}$.

	\item \label{useless-1}  If  
  $D(\A)+\ker \rho_1^2=D(\A_1)$, then $c_1=c_2$, $|\A_{H_1\cap H_2}|>2$, $D(\A)+\ker \rho_2^1=D(\A_2)$,  and a minimal generator set for $D(\A_{1,2})$ can be expressed as $\{\theta_1, \ldots, \theta_\ell, \varphi_1 \mid \theta_{{\ell},1}=\alpha_2\varphi_1\}$.
  
\item \label{useless-3}    If $D(\A)+\ker \rho_2^1\subsetneq D(\A_2)$, then  the set $\{\theta_1, \ldots, \theta_\ell, \theta_{{\ell},1}, \theta_{{\ell},2} \}$ 
forms, or can be extended to, a minimal generator set for
$D(\A_{1,2})$.

Moreover,  if $|\A_{H_1\cap H_2}|=2$, the set $\{\theta_1, \ldots, \theta_\ell, \theta_{{\ell},1}, \theta_{{\ell},2} \}$ 
forms a minimal generator set for
$D(\A_{1,2})$.
 \end{enumerate}
\end{proposition}

\begin{proof}
{Note that $c_1 \leq c_2$, which means this statement is not symmetric with respect to $\A_1$ and $\A_2$.}
We prove this statement case by case.

  \begin{enumerate}[label=(\arabic*)]
 \item \label{pf-eq} $D(\A)+\ker \rho_1^2\subsetneq D(\A_1)$ and $D(\A)+\ker \rho_2^1= D(\A_2)$.
 
 By \cref{cond_for_inter>2}, we have \(|\A_{H_1 \cap H_2}| > 2\). Consequently, \(|\A^2| = |\A_1^2|\). By \cref{level_elt_with_ker} \ref{useless_1'}, we may assume that \(\theta_{\ell, 2} = \alpha_1 \varphi_2\), where \(\varphi_2 \in D(\A_{1,2})\).
By \cref{diml}, this implies that 
\[\deg \varphi_2 = \deg \theta_{\ell, 2} - 1 = (| \A_2| - | \A^2|) - 1 = |\A_{1,2}| - | \A_1^2|.\]
Note that \(\varphi_2 \notin D(\A_1)\). 
{If it were, then \(\alpha_1 \varphi_2 \in D(\A)\), which, since it is the level element of \(D(\A_2)\), is not in \(D(\A)\).}

According to \cref{sum} and since \(\varphi_2 \in D(\A_{1,2}) \setminus D(\A_1)\), this indicates that \(D(\A_{1,2}) = D(\A_1) + S\varphi_2\). Hence, the set \(T = \{\theta_1, \ldots, \theta_\ell, \theta_{{\ell}, 1}, \varphi_2\}\) generates \(D(\A_{1,2})\).

{If \(T\) is not a minimal generating set, there exists an element \(\gamma \in T\) such that \(\gamma\) is generated by \(T \setminus \{\gamma\}\).
If \(\gamma = \theta_{\ell, 1}\), then there is some \(h \in S\) such that \(\theta_{\ell, 1} - h\varphi_2 \in D(\mathcal{A})\). 
Since \(\theta_{\ell, 1} \in D(\mathcal{A}_1)\) and \(D(\mathcal{A}) \subset D(\mathcal{A}_1)\), we can deduce that \(h\varphi_2 \in D(\mathcal{A}_1)\). However, we have claimed that \(\varphi_2 \notin D(\mathcal{A}_1)\). Therefore, it follows that \(\alpha_2 \mid h\), and thus \(h\varphi_2 \in \ker \rho_1^2\). 
By \(\theta_{\ell, 1} - h\varphi_2 \in D(\mathcal{A})\) and \(D(\mathcal{A}_1) = D(\mathcal{A}) + S\theta_{\ell, 1}\), we have
\[ D(\mathcal{A}_1) = D(\mathcal{A}) + S h \varphi_2. \]
Since \(h\varphi_2 \in \ker \rho_1^2\), we then have
\[ D(\mathcal{A}_1) = D(\mathcal{A}) + \ker \rho_1^2, \]
which contradicts our assumption.
If \(\gamma = \varphi_2\), then \(D(\A_{1,2}) = D(\A) + S\theta_{{\ell}, 1} = D(\A_1)\), which contradicts \(D(\A_{1}) \subsetneq D(\A_{1,2})\).
If there is a \(k \in \{1, \ldots, \ell\}\) such that \(\gamma = \theta_k\), then we may assume that
\begin{align}\label{000}
    0 = p_1 \theta_1 + \cdots + p_k \theta_k + \cdots + p_{\ell} \theta_{\ell} + u \theta_{{\ell}, 1} + v \varphi_2,
\end{align}
where \(p_k = 1\).
It follows that \(v \varphi_2 \in D(\A_1)\). 
Note that we claimed that \(\varphi_2 \notin D(\A_1)\). Hence \(\alpha_2 \mid v\), and \(v = \alpha_2 v'\), where \(v'\in S\).
Since \(p_k = 1\) and \(\theta_1, \cdots, \theta_{\ell}, \theta_{{\ell}, 1}\) form a minimal generating set of \(D(\A_1)\), we have \(v \neq 0\). Thus \(\deg \theta_k = \deg v + \deg \varphi_2 \geq  1 + \deg \varphi_2=c_2 \).
If \(\deg \theta_k > c_2\), then \(q_k = 0\) and \(\alpha_2\varphi_2\in S\theta_1 + \cdots + \hat{S\theta_k} + \cdots + S\theta_{\ell} + S\theta_{{\ell}, 1}\). Since \(v = \alpha_2 v'\), by \cref{000}, it follows that \(\theta_k \in S\theta_1 + \cdots + \hat{S\theta_k} + \cdots + S\theta_{\ell} + S\theta_{{\ell}, 1}\), which is a contradiction.
Thus, \(\deg \theta_k = c_2\). We may let \(v=\alpha_2\). If \(u = 0\), we have \(\alpha_2 \varphi_2 \in D(\A)\). Then \(\varphi_2 \in D(\A_2)\), which contradicts the fact that \(\alpha_1 \varphi_2\) is a level element of \(D(\A_2)\). Hence, \(u \neq 0\). If \(c_1 = c_2\), then we may assume that \(u = 1\). Thus, \(\alpha_2 \varphi_2\) can be a level element for \(D(\A_1)\). By \cref{level_elt_with_ker} \ref{useless_1}, we have \(D(\A) + \ker \rho_1^2 = D(\A_1)\), which contradicts our assumption.

As a conclusion, if \(T\) is not a minimal generating set, we may set \(k = \ell\). Then, we have \(c_1 < c_2 = \deg \theta_{\ell}\), and \(T \setminus \{\theta_{\ell}\}\) forms a minimal generating set for \(D(\A_{1,2})\). Since $\alpha_1\theta_{\ell,1}\in D(\A)$, we have \(\alpha_1\alpha_2 \varphi_2  \in S \theta_1 + \cdots + S \theta_{\ell-1}  + \K^*\alpha_1\theta_{\ell} \).

}
 \item \label{pf-eq1} $D(\A)+\ker \rho_1^2= D(\A_1)$.

By \cref{level_elt_with_ker} \ref{useless_1}, we may assume that \(\theta_{\ell, 1} = \alpha_2 \varphi_1\), where \(\varphi_1 \in D(\A_{1,2})\). This implies that \(\deg \varphi_1 = c_1 - 1\). If \(\alpha_1 \varphi_1 \in D(\A)\), we can conclude, based on \cref{elt_inA_belong_to_B}, that \(\varphi_1 \in D(\A_1)\), which contradicts the assertion that \(\theta_{\ell, 1} = \alpha_2 \varphi_1\) is a level element for \(D(\A_1)\). Therefore, we must conclude that \(\alpha_1 \varphi_1 \in D(\A_2)_{\leq c_1} \setminus D(\A)\). It follows that \(c_1 = c_2\). Hence, \(\alpha_1 \varphi_1\) can serve as a level element for \(D(\A_2)\). According to \cref{level_elt_with_ker} (2), this implies \(D(\A) + \ker \rho_2^1 = D(\A_2)\).

Referring to \cref{cond_for_inter>2}, we establish that \(|\A_{H_1 \cap H_2}| > 2\). Analogous to the proof in \cref{pf-eq}, we conclude that \(D(\A_{1,2}) = D(\A_1) + S \varphi_2\) and  \(\theta_{\ell, 2} = \alpha_1 \varphi_2\) is a level element for \(D(\A_2)\). Since both \(\alpha_1 \varphi_1\) and \(\alpha_1 \varphi_2\) can be level elements of \(D(\A_2)\), it follows that there exists a \(k \in \K \setminus \{0\}\) such that \(k \alpha_1 \varphi_1 - \alpha_1 \varphi_2 \in D(\A)\), which implies \(k \varphi_1 - \varphi_2 \in D(\A_1)_{< c_1} \subset D(\A)\). Thus, \(\varphi_2 \in S \{\theta_1, \ldots, \theta_\ell, \varphi_1\}\).
Note that \(D(\A_1) \subset S \{\theta_1, \ldots, \theta_\ell, \varphi_1\}\). Consequently, the set \(T = \{\theta_1, \ldots, \theta_\ell, \varphi_1\}\) generates \(D(\A_{1,2})\).

{
 If \(T\) is not a minimal generating set, there exists an element \(\gamma \in T\) such that \(\gamma\) is generated by \(T \setminus \{\gamma\}\). Then \(D(\A_{1,2})\) is free. Note that \(\deg \theta_1 + \cdots + \deg \theta_{\ell} = |\A| > |\A_{1,2}|\). By \cref{saito}, there is a \(k \in \{1, \ldots, \ell\}\) such that \(\gamma = \theta_k\). Suppose
\begin{align*}
    \theta_k = p_1 \theta_1 + \cdots + p_{k-1} \theta_{k-1} + p_{k+1} \theta_{k+1} + \cdots + p_{\ell} \theta_{\ell} + p \varphi_1.
\end{align*}
Note that \(\varphi_1\) is not in either \(D(\A_1)\) or \(D(\A_2)\). It follows that \(\alpha_1 \alpha_2 \mid p\). Note that \(\theta_{\ell, 1} = \alpha_2 \varphi_1\). 
Thus, \(\theta_k \in S \theta_1 + \cdots + S \theta_{k-1} + S \theta_{k+1} + \cdots + S \theta_{\ell} + S \theta_{\ell, 1}\), which contradicts the fact that \(\theta_1, \ldots, \theta_k, \ldots, \theta_{\ell}, \theta_{\ell, 1}\) form a minimal generating set of \(D(\mathcal{A}_1)\).

}

\item  $D(\A)+\ker \rho_2^1\subsetneq D(\A_2)$.

By the proof of \cref{pf-eq1}, we can conclude that \(D(\A) + \ker \rho_2^1 \subsetneq D(\A_2)\) implies \(D(\A) + \ker \rho_1^2 \subsetneq D(\A_1)\). 
Obviously, the set \(\{\theta_1, \ldots, \theta_\ell, \theta_{\ell,1}, \theta_{\ell,2}\}\) forms, or can be extended to form, a generating set for \(D(\A_{1,2})\). 
Suppose that the set \(T = \{\theta_i, \gamma_j, \theta_{\ell,1}, \theta_{\ell,2} \mid i = 1, \ldots, \ell, j = 1, \ldots, t\}\) forms a generating set for \(D(\A_{1,2})\) such that \(\gamma_j \notin S (T \setminus \{ \gamma_j\})\). Thus, \(\gamma_j\) is not in either \(D(\A_1)\) or \(D(\A_2)\). This implies that \(\deg \gamma_j \geq c_2\).

{
If \(T\) is not minimal, there exists an element \(\gamma \in T\) such that \(\gamma\) is generated by \(T \setminus \{\gamma\}\). Assume that
\begin{align}\label{useless_dependent}
    \gamma = p_1 \theta_1 + \cdots + p_{\ell} \theta_{\ell} + p \theta_{\ell,1} + q \theta_{\ell,2} + q_1 \gamma_1 + \cdots + q_t \gamma_t,
\end{align}
where $p_i,q_j,p,q\notin \K^*.$
If \(\gamma = \theta_{\ell,1}\), then we may assume that \(p = 0\). Note that \(\deg \gamma_j \geq c_2 \geq c_1 = \deg \theta_{\ell,1}\). Thus, \(q_j \in \mathbb{K}\). It follows that \(q_j = 0\) for all \(j = 1, \ldots, t\). Therefore, \(\theta_{\ell,1} \in D(\mathcal{A}_2)\), which is a contradiction.

Similarly, we can conclude that \(\gamma \neq \theta_{\ell,2}\).

If there is a \(k \in \{1, \ldots, \ell\}\) such that \(\gamma = \theta_k\), then we may assume that \(p_k = 0\). 
If \(q_j = 0\) for all \(j\), then \(pq \neq 0\). By \cref{useless_dependent}, we have \(\deg \theta_k = \deg q + \deg \theta_{\ell,2} > c_2\). Thus, \(p_i = 0\) for any \(\deg \theta_i \geq \deg \theta_k\). By the definition of an SPOG arrangement, both \(\theta_{\ell,1}\) and \(\theta_{\ell,2}\) are \(S\)-dependent with \(\{\theta_i \mid \deg \theta_i \leq c_2\}\). It follows that \(\theta_k\) is \(S\)-dependent with \(\{\theta_i \mid \deg \theta_i < \deg \theta_k\}\), which is a contradiction.
Thus, we may assume that \(q_1\) has the smallest degree among the non-zero coefficient \(\{q_j\}\).  Therefore, \(\deg \theta_k = \deg q_1 + \deg \gamma_1 > \deg \gamma_1\).
Since \(\alpha_1 \gamma_j \in D(\mathcal{A}_1)\) and \(\gamma_j\) is not in \(D(\mathcal{A}_1)\), it follows that \(\gamma_j\) is \(S\)-dependent with \(\{\theta_{\ell,1}\} \cup \{\theta_i \mid \deg \theta_i \leq \deg \gamma_j\}\). 
Note that \(\deg \gamma_j \geq c_2\) for all \(j\), and \(\theta_{\ell,1}\) is \(S\)-dependent with \(\{\theta_i \mid \deg \theta_i \leq c_2\}\).
It follows that \(\theta_k\) is \(S\)-dependent with \(\{\theta_i \mid \deg \theta_i \leq \deg \gamma_1\}\), which is a contradiction.

In conclusion, the set \(\{\theta_1, \ldots, \theta_\ell, \theta_{\ell,1}, \theta_{\ell,2}\}\) forms, or can be extended to, a minimal generating set for \(D(\mathcal{A}_{1,2})\). Furthermore, in the scenario where \(|\mathcal{A}_{H_1 \cap H_2}| = 2\), as corroborated by \cref{inters=2}, the set \(\{\theta_1, \ldots, \theta_\ell, \theta_{\ell,1}, \theta_{\ell,2}\}\) serves as a minimal generating set for \(D(\mathcal{A}_{1,2})\).}
\end{enumerate}

\end{proof}



When constructing a graded free resolution, it is minimal if, and only if, at each step, we select a minimal homogeneous system of generators for the kernel of the differential. Refer to Construction 4.2 and Theorem 7.3 in \cite{min_resol_construct} for details.
Let us employ this approach to construct a minimal free resolution for $D(\A_{1,2})$.

\begin{theorem}\label{free_resolution_Aij}
 If $|\DS{\A_{1,2}}|\leq \ell +2$, then $D(\A_{1,2})$ has a minimal free resolution of one of the following forms:
\begin{enumerate}[label=(\arabic*)]

 \item\label{useless-case2-1} 
     If $D(\A)+\ker \rho_1^2=D(\A_1)$,
    there exist $f_1,\ldots,f_{\ell}\in S$ such that
     \begin{align*} 
        0 \rightarrow S[-c_1-1] \xrightarrow{(\alpha_1\alpha_2,f_1,\ldots,f_{\ell})} S[-c_1+1] \oplus \left(\bigoplus_{i=1}^{\ell} S[-d_i]\right) \rightarrow D(\A_{1,2}) \rightarrow 0.
    \end{align*}

    \item\label{useless-case2-2}    If $D(\A)+\ker \rho_1^2\subsetneq D(\A_1)$, $D(\A)+\ker \rho_2^1= D(\A_2)$, and $|\DS{\A_{1,2}}|= \ell +2$, there exist $f_i,g_i,g\in S$ with $g\neq 0$ such that
    \begin{align*}
        0 \rightarrow S[-c_1-1] \oplus S[-c_2]  & \xrightarrow{\begin{array}{c} (\alpha_1,0, f_1,\ldots,f_{\ell}) \\ (g,\alpha_2, g_1,\ldots,g_{\ell}) \end{array}} \\ & S[-c_1] \oplus S[-c_2+1] \oplus \left(\bigoplus_{i=1}^{\ell} S[-d_i]\right)
        \rightarrow D(\A_{1,2}) \rightarrow 0.
    \end{align*}

 \item \label{useless-case2-2'} If $D(\A)+\ker \rho_1^2\subsetneq D(\A_1)$, $D(\A)+\ker \rho_2^1= D(\A_2)$, and $|\DS{\A_{1,2}}|= \ell +1$, there exist $f_1,\ldots,f_{\ell-1}\in S$ such that
   \begin{align*}
        0 \rightarrow S[-c_1-1] \xrightarrow{(\alpha_1,0, f_1,\ldots,f_{\ell-1})}  S[-c_1] \oplus S[-c_2+1] \oplus \left(\bigoplus_{i=1}^{\ell-1} S[-d_i]\right)
        \rightarrow D(\A_{1,2}) \rightarrow 0.
    \end{align*}
   Moreover, in this case, $\A_{1,2}$ is SPOG.
   
    \item \label{useless-case2-3}   If $D(\A)+\ker \rho_2^1\subsetneq D(\A_2)$, 
there exist  $f_i,g_i\in S$ such that
    \begin{align*}
        0 \rightarrow S[-c_1-1] \oplus S[-c_2-1] &\xrightarrow{\begin{array}{c} (\alpha_1,0, f_1,\ldots,f_{\ell}) \\ (0,\alpha_2, g_1,\ldots,g_{\ell}) \end{array}}  \\ &  S[-c_1] \oplus S[-c_2] \oplus \left(\bigoplus_{i=1}^{\ell} S[-d_i]\right) \rightarrow D(\A_{1,2}) \rightarrow 0.
    \end{align*}
\end{enumerate}
  \end{theorem}

\begin{proof}
We retain the notation introduced in \cref{min_gen_set} and its proof.
Given $|\DS{\A_{1,2}}|\leq \ell +2$ and utilizing \cref{min_gen_set}, we proceed to analyze the proof on a case-by-case basis.

    \begin{enumerate}[label=(\arabic*)]

    \item $D(\A)+\ker \rho_1^2=D(\A_1)$.
    
Note that $\{ \varphi_1, \theta_i \mid i=1, \ldots, \ell \}$ is a minimal generator set for $D(\A_{1,2})$. Since $\theta_{\ell ,1} = \alpha_2\varphi_1$ is a level element for $D(\A_1)$, 
we may assume that 
$\alpha_1(\alpha_2\varphi_1)+\sum_{i=1}^\ell f_i\theta_i=0$.
 In other words, there exists a relation
 \begin{align}\label{useless-rel01}
(\alpha_1\alpha_2,f_1,\ldots,f_\ell)
\end{align}
 between $\{\varphi_1, \theta_1,   \ldots, \theta_\ell \}$.
 
Suppose there exists another $S$-independent relation, say:
 \begin{align}\label{useless-rel02}
(p,p_1,\ldots,p_\ell).
\end{align}
This implies that $p\varphi_1 \in D(\A)$.
Consequently,  $p\varphi_1(\alpha_2)\in S\alpha_2$.
If $\varphi_1(\alpha_2)\in S\alpha_2$, it implies that $\varphi_1\in D(\A_1)$, which contradicts the statement that 
$\theta_{\ell ,1} = \alpha_2\varphi_1$ is a level element for $D(\A_1)$.
As a result, we conclude that $\alpha_2 | p$. 
{Therefore Relation  \eqref{useless-rel02}   is actually a relation amongst the minimal generators of the SPOG arrangement $\A_1$, of which there is only one by \cref{spogarr}.}
 Consequently, there exist $f_1, \ldots, f_{\ell}\in S$  such that $D(\A_{1,2})$ has a minimal free resolution of the following forms:
  \begin{align*} 
        0 \rightarrow S[-c_1-1] \xrightarrow{(\alpha_1\alpha_2,f_1,\ldots,f_{\ell})} S[-c_1+1] \oplus \left(\bigoplus_{i=1}^{\ell} S[-d_i]\right) \rightarrow D(\A_{1,2}) \rightarrow 0.
    \end{align*}

\item $D(\A)+\ker \rho_1^2\subsetneq D(\A_1)$, $D(\A)+\ker \rho_2^1=D(\A_2)$ and $|\DS{\A_{1,2}}|=\ell+2$.
        
Note that $\{\theta_{\ell ,1}, \varphi_2,\theta_1,   \ldots, \theta_\ell \}$ is a minimal generator set for $D(\A_{1,2})$. Since $\theta_{\ell ,1}$ is a level element for $D(\A_1)$, we can assume that
\begin{align*}
\alpha_1\theta_{\ell ,1} + \sum_{i=1}^\ell f_i\theta_i = 0.
\end{align*}
In other words, there exists a relation
\begin{align}\label{useless-relation1}
(\alpha_1, 0, f_1, \ldots, f_\ell)
\end{align}
between $\{ \theta_{\ell ,1}, \varphi_2,\theta_1,   \ldots, \theta_\ell \}$.

Since $\alpha_2\varphi_2\in D(\A_1)\setminus D(\A)$, we may assume that 
\begin{align*}
       g\theta_{\ell ,1}+ \alpha_2\varphi_2+\sum_{i=1}^\ell g_i\theta_i=0,
\end{align*}
with $g\neq 0$.
This {gives another relation}
\begin{align}\label{useless-relation2}
        ( g,\alpha_2,g_1,\ldots, g_\ell)
\end{align}
between $\{ \theta_{\ell ,1}, \varphi_2,\theta_1,   \ldots, \theta_\ell \}$,  and it is $S$-independent with Relation \eqref{useless-relation1}.

If there exist additional relations, denoted as $(p, q, p_1, \ldots, p_\ell)$, the implication is as follows:
$$p\theta_{\ell ,1}+q\varphi_2+\sum_{i=1}^\ell p_i\theta_i=0.$$

This implies that $q\varphi_2\in D(\A_1)$.
Consequently,  $q\varphi_2(\alpha_2)\in S\alpha_2$.
If $\varphi_2(\alpha_2)\in S\alpha_2$, it implies that $\alpha_1\varphi_2\in D(\A)$, which contradicts the assertion that the level element 
$\theta_{\ell ,2}=\alpha_1\varphi_2$ for $D(\A_2)$ is not in $D(\A)$.
As a result, we conclude that $\alpha_2 | q$. 
{We see $\frac{q}{\alpha_2}(g, \alpha_2, g_1, \ldots, g_\ell)-(p, q, p_1, \ldots, p_\ell)$ is a relation of the form \eqref{useless-relation1}, which is the unique $D(\A_1)$ relation.
Thus, $(p, q, p_1, \ldots, p_\ell)$ can be represented by the relations \eqref{useless-relation1} and \eqref{useless-relation2}.
That \eqref{useless-relation1} and \eqref{useless-relation2} are $S$-independent is clear since $\alpha_2\neq 0$.}

Thus $D(\A_{1,2})$ has a minimal free resolution of the following forms:
    \begin{align*}
        0 \rightarrow S[-c_1-1] \oplus S[-c_2]  & \xrightarrow{\begin{array}{c} (\alpha_1,0, f_1,\ldots,f_{\ell}) \\ (g,\alpha_2, g_1,\ldots,g_{\ell}) \end{array}} \\ & S[-c_1] \oplus S[-c_2+1] \oplus \left(\bigoplus_{i=1}^{\ell} S[-d_i]\right)
        \rightarrow D(\A_{1,2}) \rightarrow 0.
    \end{align*}

\item $D(\A)+\ker \rho_1^2\subsetneq D(\A_1)$, $D(\A)+\ker \rho_2^1=D(\A_2)$ and $|\DS{\A_{1,2}}|=\ell+1$.

{
Since \(\A_1\) is SPOG and \(D(\A_1)\) is generated by \(\{\theta_{{\ell},1}, \theta_1, \ldots, \theta_{\ell}\}\), we may assume that
\begin{align*}
    \alpha_1 \theta_{\ell,1} + \sum_{i=1}^\ell f_i \theta_i = 0.
\end{align*}
By \cref{min_gen_set} \ref{useless-2}, we have \(c_1 < c_2 = \deg \theta_{\ell}\).
Thus, \(f_{\ell} = 0\), and we obtain a relation
\[
(\alpha_1, 0, f_1, \ldots, f_{\ell-1})
\]
between { \(\{\theta_{{\ell},1}, \varphi_2, \theta_1, \ldots, \theta_{\ell-1} \}\),} the minimal generating set of \(D(\A_{1,2})\).}

If additional relations are present, denoted as \((p, q, p_1, \ldots, p_{\ell-1})\), the corresponding equation is expressed as follows:
\[
p \theta_{\ell,1} + q \varphi_2 + \sum_{i=1}^{\ell-1} p_i \theta_i = 0.
\]
{Note that \(\theta_{\ell,1}\) is \(S\)-dependent with \(\theta_1, \ldots, \theta_{\ell-1}\) by $f_\ell=0$.}
By rearranging, we get:
\[
-q \varphi_2 = p \theta_{\ell,1} + \sum_{i=1}^{\ell-1} p_i \theta_i,
\]
which implies that \(q \varphi_2\) is \(S\)-dependent with \(\theta_1, \ldots, \theta_{\ell-1}\).
However, considering \(q \alpha_2 \theta_{\ell,2} = q \alpha_2 \alpha_1 \varphi_2 \in S \theta_1 + \cdots + S \theta_{\ell-1} + q \K^* \alpha_1 \theta_\ell\), we deduce that \(q = 0\).
This gives a relation between the generators for the SPOG arrangement \(\A_1\). Hence, we can conclude that
\(\{\theta_{\ell,1}, \varphi_2, \theta_1, \ldots, \theta_{\ell-1}\}\) has a unique relation
\((\alpha_1, 0, f_1, \ldots, f_{\ell-1})\).
Thus, \(D(\A_{1,2})\) has a minimal free resolution of the following form:
\begin{align*}
    0 \rightarrow S[-c_1-1] \xrightarrow{(\alpha_1, 0, f_1, \ldots, f_{\ell-1})} S[-c_1] \oplus S[-c_2+1] \oplus \left(\bigoplus_{i=1}^{\ell-1} S[-d_i]\right) \rightarrow D(\A_{1,2}) \rightarrow 0.
\end{align*}
Moreover, based on the definition of SPOG, we deduce that \(\A_{1,2}\) is SPOG.

    \item  $D(\A)+\ker \rho_2^1\subsetneq D(\A_2)$.
    
Note that $\{\theta_{\ell ,1}, \theta_{\ell ,2},  \theta_1, \ldots, \theta_{\ell} \}$ is a minimal generator set for $D(\A_{1,2})$. Since $\theta_{\ell, j}$, where $j=1,2$, is a level element for $D(\A_j)$, we can assume that
\begin{align*}
&\alpha_1\theta_{\ell ,1} + \sum_{i=1}^\ell f_i\theta_i = 0,\\
&\alpha_2\theta_{\ell ,2} + \sum_{i=1}^\ell g_i\theta_i = 0.
\end{align*}

In other words, there are two relations
\begin{align}
&(\alpha_1, 0, f_1, \ldots, f_\ell)\label{useless-relationI-1}
\\
&( 0, \alpha_2,{g_1, \ldots, g_\ell})\label{useless-relationI-2}
\end{align}
between $\{ \theta_{\ell ,1}, \theta_{\ell ,2}, \theta_1, \ldots, \theta_{\ell} \}$. 
If there exist additional relations, denoted as $(p, q, p_1, \ldots, p_\ell)$, the implication is as follows:
$$p\theta_{\ell ,1}+q\theta_{\ell ,2}+\sum_{i=1}^\ell p_i\theta_i=0.$$

This implies that $q\theta_{\ell ,2}\in D(\A_1)$. 
Consequently,  $q\theta_{\ell ,2}(\alpha_2)\in S\alpha_2$.
Considering that 
$\theta_{\ell ,2}$ is a level element for $D(\A_2)$, it follows that $\theta_{\ell ,2}(\alpha_2)\notin S\alpha_2$.
As a result, we conclude that $\alpha_2 | q$. According to Relation \eqref{useless-relationI-2}, we have $q\theta_{\ell ,2}\in D(\A)$.
This observation suggests that this relation establishes a connection among $\{\theta_{\ell ,1},  \theta_1, \ldots, \theta_{\ell}  \}$. 
Given that Relations \eqref{useless-relationI-1} and  \eqref{useless-relationI-2} uniquely characterize the relationship between sets $\{\theta_{\ell ,1},  \theta_1, \ldots, \theta_{\ell}  \}$ and $\{\theta_{\ell ,2},  \theta_1, \ldots, \theta_{\ell}  \}$, respectively, and taking into account the $S$-independence of $\theta_1, \ldots, \theta_\ell$, we can deduce that Relation $(p, q, p_1, \ldots, p_\ell)$ can be represented by Relations \eqref{useless-relationI-1} and  \eqref{useless-relationI-2}.
That Relations \eqref{useless-relationI-1} and  \eqref{useless-relationI-2} are $S$-independent is clear since $\alpha_1,\alpha_2\neq 0$.

Thus $D(\A_{1,2})$ has a minimal free resolution of the following forms:
 \begin{align*}
        0 \rightarrow S[-c_1-1] \oplus S[-c_2-1]  & \xrightarrow{\begin{array}{c} (\alpha_1,0, f_1,\ldots,f_{\ell}) \\ (0,\alpha_2, g_1,\ldots,g_{\ell}) \end{array}} \\ & S[-c_1] \oplus S[-c_2] \oplus \left(\bigoplus_{i=1}^{\ell} S[-d_i]\right)
        \rightarrow D(\A_{1,2}) \rightarrow 0.
    \end{align*}
\end{enumerate}
\end{proof}

\begin{theorem}\label{free_resolution_>l+2_M0}
 Let $\mathcal{B}$ in \cref{def_resol} be defined as $\mathcal{A}_{1,2}$.
 Assume that $|\DS{\A_{1,2}}| > \ell + 2$ with $\DS{\A_{1,2}} = (d_1,\ldots,d_\ell,c_1,\ldots,c_r)$.
 Then, the following statements hold:  
    \begin{enumerate}[label=(\arabic*)]
        \item We have
        $\bigoplus_{j=1}^r S[-c_j-1] \subsetneq M_1$, and there exist $h_{ij}, h_j\in S$ with $h_j\neq 0$ such that
          \begin{align*}
             & R_1(1) = (h_{11},\ldots,h_{\ell 1},\alpha_1,0,\ldots,0),\\
             &R_1(2)= (h_{12},\ldots,h_{\ell 2},0,\alpha_2,0\ldots,0),\\
             &R_1(j)= (h_{1j},\ldots,h_{\ell j},h_j,0,\ldots,0,\alpha_2,0,\ldots,0),
         \end{align*}
     where $R_1(j)$ denotes the $j$-th row of $R_1$,
     and 
     $R_1(j)(\ell + j)=\alpha_2$ for $j=3,\ldots,r$.

     \item  We write
\begin{align*}
 & R_1 =(\psi_1,\ldots, \psi_r, \phi_1,\ldots, \phi_t)^T,  \\
 & M_1 = \left(\bigoplus_{j=1}^r S[-c_i-1]\right) \oplus \left(\bigoplus_{i=1}^t S[-e_i]\right)
\end{align*}
with some $t>0$. Here, $\psi_j$ corresponds to $R_1(j)$ as defined in $(1)$, where $j=1,\cdots,r$.

Then, 
$\bigoplus_{i=1}^t S[-e_i-1] \subset M_2$.
In particular, $M_2\neq 0$.
Moreover, 
$R_2(i)$ has the form $(f_{1i},\ldots,f_{r i},0,\ldots,0,\alpha_2,0,\ldots,0)$,
where  $f_{ji}\in S$ and $R_{2}(i)(r+i)=\alpha_2$ for $i=1,\ldots, t$.
When $\bigoplus_{i=1}^t S[-e_i-1] = M_2$,
we have $t=r-2$ and $M_3=0$.
\end{enumerate} 
\end{theorem}

\begin{proof}

    \begin{enumerate}[label=(\arabic*)]
        \item
        
     By \cref{min_gen_set}, we may assume that the minimal generator set of $D(\A_{1,2})$ is $T=\{\theta_i,\varphi_j\mid \deg \theta_i=d_i,\deg \varphi_j=c_j, i=1,\ldots,\ell, j=1,\ldots,r\}$. Here $\{\theta_i\mid i=1,\ldots,\ell\}$ forms a basis for $D(\A)$, and $\varphi_1$ and $\varphi_2$ represent the level elements  of $D(\A_{1})$ and $D(\A_2)$, respectively. By \cref{min_gen_set}, it holds that $c_i \geq c_2$ for $i > 2$, and $$M_0=\left(\bigoplus_{i=1}^\ell S[-d_i]\right)\bigoplus \left(\bigoplus_{j=1}^r S[-c_i]\right).$$

Note that $\alpha_1\varphi_1, \alpha_2\varphi_2\in  D(\A)$ and $\alpha_2\varphi_j\in D(\A_1)\setminus D(\A)$ for $j>2$. Hence we have the following relations:
     \begin{align*}
         & \psi_1:=(h_{11},\ldots,h_{\ell 1},\alpha_1,0,\ldots,0),\\
         &\psi_2:=(h_{12},\ldots,h_{\ell 2},0,\alpha_2,0\ldots,0),\\
         &\psi_j:=(h_{1j},\ldots,h_{\ell j},h_j,0,\ldots,0,\alpha_2,0,\ldots,0),
     \end{align*}
where $h_j \neq 0$ and $\alpha_2=\psi_j(\ell+j)$. 

Suppose there exists a  $j_0 \in \{1, \ldots, r\}$ such that {$\psi_{j_0}$ is not $\K$-spanned by $\{R_1(i)\}_i$}. In this scenario, we can assume the existence of relations $R_{1}(1), \ldots, R_1(w) \in R_1$ such that $\psi_{j_0} = \sum_{i=1}^w p_iR_1(i)$, where $p_i \in S_{\geq 1}$. It's noteworthy that $p_i$ and $R_1(i)(j)$, where $i=1, \ldots, w$ and $j=1, \ldots, \ell +r$, are homogeneous polynomials in $S_{\geq 1}$.

If $j_0=1$, this implies that
\[
\alpha_1 = \psi_{1}(\ell +1) = \sum_{i=1}^w p_iR_1(i)(\ell +1) \in S_{\geq 2},
\]
leading to a contradiction.
Hence, we deduce that $j_0 > 1$. Further, examining the expression
\[
\alpha_2 = \psi_{j_0}(\ell +j_0) = \sum_{i=1}^w p_iR_1(i)(\ell +j_0) \in S_{\geq 2},
\]
we again arrive at a contradiction. Consequently, we have that $R_1(1)=\psi_1, \ldots, R_1(r)=\psi_r$ and $\bigoplus_{j=1}^r S[-c_i-1] \subset M_1$.

Assume $\bigoplus_{j=1}^r S[-c_i-1] = M_1$. 
Observe that $\psi_1(\ell+1, \ldots, \ell+r), \ldots, \psi_r(\ell+1, \ldots, \ell+r)$ form a lower triangular matrix. It is evident that the relations $\psi_1, \ldots, \psi_r$ are $S$-independent over $S$.
As a consequence, we deduce that 
$M_2 = 0$.
By referencing \cref{betti-1}, we can deduce that $ |\A_{1,2}|=(\sum_{i=1}^\ell d_i + \sum_{j=1}^r c_j) - (\sum_{j=1}^r (c_j + 1)) =\sum_{i=1}^\ell d_i-r$.
Considering $|\A_{1,2}| = |\A| - 2 = \sum_{i=1}^\ell d_i - 2$, it follows that $r = 2$, contradicting the assumption $|\DS{\A_{1,2}}| = \ell + r > \ell + 2$. Hence, we conclude that 
$\bigoplus_{j=1}^r S[-c_i-1] \subsetneq M_1$.

\item 
{
 We can assume that $\phi_i = (u_{1i}, \ldots, u_{\ell i}, v_{1i}, \ldots, v_{r i})$. We set $\phi$ as follows:
\begin{align*}
    \phi =& \alpha_2\phi_i - (v_{2i}\psi_2 + \cdots + v_{r i}\psi_r)\\
    =& (\alpha_2u_{1i} + \sum_{j=2}^r v_{ji}h_{1j}, \ldots, \alpha_2u_{\ell i} + \sum_{j=2}^r v_{ji}h_{\ell j}, \alpha_2v_{1i} + \sum_{j=3}^r v_{ji}h_j, 0, \ldots, 0).
\end{align*}
If \(\phi = 0\), let \(f_{1i} = 0\) and \(f_{ji} = -v_{ji}\) for \(j = 2, \ldots, r\). Then we have
\[ 
R_{\phi_i} = (f_{1i}, \ldots, f_{r i}, 0, \ldots, 0, \alpha_2, 0, \ldots, 0) 
\]
which represents a relation among the relations \(R_1(1), \ldots, R_1(r+t)\), where \(\alpha_2 = R_{\phi_i}(r+i)\).
Otherwise, \(\phi\) forms a relation among \(\theta_1, \ldots, \theta_\ell, \varphi_1\). Since \(\mathcal{A}_1\) is SPOG, \(\psi_1\) is the unique relation among \(\theta_1, \ldots, \theta_\ell, \varphi_1\). Hence, we deduce the existence of some \(v \in S\) such that \(\phi - v\psi_1 = 0\). Let \(f_{1i} = -v\) and \(f_{ji} = -v_{ji}\) for \(j = 2, \ldots, r\). It follows that
\[ 
R_{\phi_i} = (f_{1i}, \ldots, f_{r i}, 0, \ldots, 0, \alpha_2, 0, \ldots, 0) 
\]
represents a relation among the relations \(R_1(1), \ldots, R_1(r+t)\), where \(\alpha_2 = R_{\phi_i}(r+i)\).

If there exists an $i_0 \in \{1, \ldots, t\}$ such that $R_{\phi_{i_0}} $ is not $\K$-spanned by $ \{R_2(i)\}_i$, we can assume the existence of relations $R_{2}(1),\ldots,R_2(w)$ such that 
$R_{\phi_{i_0}} = \sum_{i=1}^w p_iR_2(i)$, where $p_i\in S_{\geq 1}$.
It's worth noting that $p_i$ and $R_2(i)(j)$, where $i=1,\ldots,w$ and $j=1,\ldots,r+t$, are homogeneous polynomials in $S_{\geq 1}$.
This implies that

\[
\alpha_2 = R_{\phi_{i_0}}(r+i_0)
= \sum_{i=1}^w p_iR_2(i)(r+i_0) \in S_{\geq 2},
\]
leading to a contradiction.
Consequently, this reasoning implies that $R_2(1)=R_{\phi_1},\ldots, R_2(t)=R_{\phi_t}$ and $\bigoplus_{j=1}^t S[-e_i-1]\subset M_2$.

 Assume  $M_2=\bigoplus_{i=1}^t S[-e_i-1]$.
Observe that $R_{\phi_1}(r+1, \ldots, r+t), \ldots, R_{\phi_t}(r+1, \ldots, r+t)$
 form a lower triangular matrix,
 it is evident that the relations $R_{\phi_1}, \ldots, R_{\phi_t}$ are $S$-independent.
As a consequence, we deduce that
$M_3 = 0$. 
 Since 
$(\sum_{i=1}^\ell d_i + \sum_{j=1}^r c_j) - (\sum_{j=1}^r (c_j + 1) + \sum_{i=1}^t e_i)+( \sum_{i=1}^t (e_i+1)) = |\A_{1,2}|$ by \cref{betti-1},
it follows that $r-t = 2$.}
\end{enumerate}
\end{proof}



\begin{proof}[Proof of \cref{proj_dim_Aij}]
By \cref{diml,free_resolution_Aij,free_resolution_>l+2_M0}, the conclusion follows immediately.
\end{proof}
We would like to show some examples\footnote{We have written a program and used it to compute examples. 
\\ 
You can find the code at  https://github.com/jcwjmz/LogarithmicDerivationModule.git}.

\begin{example}
    Let $\ell=4$ and 
\begin{align*}
Q(\A)=&x_1x_2 x_3 x_4 (x_1-x_2) (x_1-x_3) (x_1-x_4)(x_2-x_3) (x_2-x_4)\\
&(x_3-x_4)(x_2-x_3+x_4)(x_1-x_2+x_3-x_4) .
\end{align*}
Then, $\A$ is free with $\EP{\A}=(1,3,4,4)$. 
The order of $H_i\in \A$ is consistent with its order of appearance in the polynomial.

By computer, $\A_1$ and $\A_2$ both are SPOG with level $4$, and  $\A_8$ and $\A_{10}$ both are SPOG with level $5$. Moreover, we have:
 \begin{enumerate}[label=(\arabic*)]
\item $\DS{\A_{1,2}}=(1,3,3,4,4)$.
The minimal free resolution of $D(\mathcal{A}_{1,2})$ has the following forms:
\begin{align*} 
        0 \rightarrow S[-5] \rightarrow 
        S[-4]^2 \oplus S[-3]^2 \oplus S[-1] \rightarrow D(\A_{1,2}) \rightarrow 0.
    \end{align*}

\item  $\DS{\A_{2,10}}=(1,3,4,4,4,4)$.
The minimal free resolution of $D(\mathcal{A}_{2,10})$ has the following forms:
\begin{align*} 
        0 \rightarrow S[-5]^2 \rightarrow 
         S[-4]^4 \oplus S[-3] \oplus S[-1] \rightarrow D(\A_{2,10}) \rightarrow 0.
    \end{align*}

\item $\DS{\A_{1,8}}=(1,3,4,4,4,5)$.
The minimal free resolution of $D(\mathcal{A}_{1,8})$ has the following forms:
\begin{align*} 
        0 \rightarrow S[-6] \oplus S[-5] \rightarrow 
        S[-5] \oplus S[-4]^3 \oplus S[-3] \oplus S[-1] \rightarrow D(\A_{1,8}) \rightarrow 0.
    \end{align*}
\end{enumerate}

\end{example}



\begin{example}
    Let $\ell=4$ and 
\begin{align*}
    Q(\A)=&x_1 x_2 x_3 x_4(x_1-x_2) (x_1-x_3)(x_2-x_3)(x_3-x_4)\\
     &(x_2-x_3+x_4) (x_1-x_2+x_3-x_4).
\end{align*}
Then, $\A$ is free with $\EP{\A}=(1,3,3,3)$. The order of $H_i\in \A$ is consistent with its order of appearance in the polynomial.

Note that $\A^{2}$ is not free, serving as a counter-example to Orlik's Conjecture. This can be found in \cite{DiPasquale} or \cite{N-S-2023} by performing a coordinate change.

By computation, $\A_1, \A_2$, and $\A_3$ are SPOG with level $3$, and $|\DS{\A_{2,j}}|=5<\ell +2$ for any $H_j\in \A_2$. Additionally, for $\A^1$ and $\A^3$ are free, we found that $|\DS{\A_{1,3}}|=7>\ell+2$.
The minimal free resolution of $D(\mathcal{A}_{1,3})$ has the following form:

\[
0 \rightarrow S[-5] \rightarrow 
S[-4]^4 \rightarrow 
S[-3]^6 \oplus S[-1] \rightarrow D(\A_{1,3}) \rightarrow 0.
\] 
\end{example}

\begin{remark}\label{rem:counter}

The counterexample to Orlik's conjecture mentioned above corresponds to the situation $|D(\A_{1,2})|>\ell+2$. However, every NT-free-2 arrangement of Edelman and Reiner's counterexample, with dimension $\ell=5$, 
contains at most $7=\ell+2$ minimal generators.
\end{remark}

Applying our theorem allows us to recover the result in \cite{abe-yama2023} and verify a part of Conjecture 4.4 from \cite{spog}.

\begin{corollary}\cite{abe-yama2023}
     If $\A$ and $\A_{1,2} $ are both free, then at least one of $\A_1$ and $\A_2$ is free.
\end{corollary}

  \begin{proof}
  This is Theorem 1.2 in \cite{abe-yama2023}. Moreover, in this paper, we can readily arrive at the same conclusion through \cref{min_gen_set}.
\end{proof}






\begin{corollary}\label{comfirm_conj}
If $\A_{1,2} $ is SPOG, then 
$c_2=d_j$ for some $j\in \{2,\ldots,l\}$.
\end{corollary}

\begin{proof}
By referring to \cref{free_resolution_Aij}, we establish its validity for Case (3) in \cref{free_resolution_Aij}, focusing solely on \cref{free_resolution_Aij}  \ref{useless-case2-1}.

Assuming $\PO{\A_{1,2}} =(1, p_2, \ldots, p_{\ell})$, \cref{coeff b_2} implies $\sum_{i=2}^{\ell} p_i = |\A_{1,2}|$. Examining \cref{free_resolution_Aij}  \ref{useless-case2-1}, we find $\DS{\A_{1,2}}=(1, d_2, \ldots, d_{\ell}, c_2-1)$. It's noteworthy that $1 + \sum_{i=2}^{\ell} d_i = |\A|$ and $\PO{\A_{1,2}} \subset \DS{\A_{1,2}}$. This implies the existence of a $d_j$ such that $\sum_{i=2}^{\ell} d_i + (c_2-1) - d_j = |\A_{1,2}|$. Consequently, we deduce that $c_2 = d_j.$

\end{proof}

\section{Three Dimensional Case}\label{sec:dim3}

In this section, we will persist in employing the notation introduced at the outset of the preceding section. Furthermore, we fix the parameter $\ell$ to be 3.

If $|\A_{H_1\cap H_2}|>2$, there exists a plane $H\in \A_{H_1\cap H_2}$ that is distinct from both $H_1$ and $H_2$. We can assume that $H=\ker x_2$, $H_1=\ker x_1$, and $H_2=\ker (x_1-x_2)$.

\begin{proposition}\label{coff-of-level-elt}
If $|\A_{H_1\cap H_2}|>2$,
    {then for any} basis $\theta_2,\theta_3$ for $D_H(\A)$, there exist a level element $\theta_{3,i}$ for $D(\A_i)$ such that 
\begin{align}
  x_1\theta_{3,1}&=f_2\theta_2+f_3\theta_3,f_2,f_3\in \K [x_2,x_3],\label{Eq_nonsense_1}\\
  (x_1-x_2)\theta_{3,2}&=g_2\theta_2+g_3\theta_3,g_2,g_3\in \K [x_2,x_3].\label{Eq_nonsense_2}
\end{align}
\end{proposition}

\begin{proof}
    Suppose that  $\theta_{3,i}'\in D_H(\A_i)$ is a level element  for $D(\A_i)$, where $i=1,2$.  
    Using  \cref{diml}, we may assume that:
\begin{align}
x_1\theta_{3,1}'&=f_2'\theta_2+f_3'\theta_3,\label{Idonotknow1}\\
  (x_1-x_2)\theta_{3,2}'&=g_2'\theta_2+g_3'\theta_3,
\end{align}
Since $f_i'$ and $g_i'$ are {homogeneous} polynomial in $S=\K[x_1,x_2,x_3]$, we may assume that:
\begin{align*}
    f_i'&=x_1f_i''+f_i,\ i=2,3,\ f_i\in \K[x_2,x_3],\\
    g_i'&=(x_1-x_2)g_i''+g_i,\ i=2,3,\ g_i\in \K[x_2,x_3].
\end{align*}
Let $\theta_{3,1}=\theta_{3,1}'-f_2''\theta_2-f_3''\theta_3$.
Considering that $\theta_{3,1}\notin D(\A)$ and $\deg \theta_{3,1}=\deg \theta_{3,1}'$, it is evident that $\theta_{3,1}$ qualifies as a level element for $D(\A_1)$.
By substituting $\theta_{3,1}$ into \cref{Idonotknow1}, we obtain \cref{Eq_nonsense_1}. A parallel application yields \cref{Eq_nonsense_2}.
\end{proof}

\begin{lemma}\label{part_pf_thm_in_dim3}
 Consider the scenario where $|\A_{H_1\cap H_2}|>2$, and let $\theta_2,\theta_3$ form a basis for $D_H(\A)$. Assuming that $\theta_{3,1}$ and $\theta_{3,2}$ are the level elements for $D(\A_1)$ and $D(\A_2)$, respectively, satisfying \cref{coff-of-level-elt}. If 
   \begin{align*}
&  \begin{cases}
 f_2=x_2f_2'+kx_3^r, {\ k\in \K^*, r\in \Z^+,}\ f_2'\in \K [x_2,x_3],\\
 g_2=x_2g_2'+k'x_3^{r'},\  {\ k'\in \K^*, r'\in \Z^+,}\ g_2'\in \K [x_2,x_3],
 \end{cases}
  \begin{cases}
 f_3=x_2f_3',\ f_3'\in \K [x_2,x_3],\\
 g_3=x_2g_3',\ { g_3'}\in \K [x_2,x_3],
\end{cases}
   \end{align*}
then 
\begin{enumerate}[label=(\arabic*)]
    \item \label{useless-lem4p2} $c_1=c_2$ if and only if
    $D(\A_1)= D(\A) + \ker \rho_{1}^2$.
    In this case, $c_1=c_2=\max\{d_1,d_2,d_3\}$ if and only if $\A_{1,2}$ is SPOG.

    \item $c_1<c_2$ if and only if $D(\A_1) \subsetneq D(\A) + \ker \rho_{1}^2$ and $D(\A_2) = D(\A) + \ker \rho_{2}^1$.
\end{enumerate}
\end{lemma}

\begin{proof}
Since $\deg \theta_{3,1}=c_1\leq \deg \theta_{3,2}=c_2$, we have
 $r'\geq r$. 
 Assuming that $h=\frac{k'}{k}x_3^{r'-r}$, then we may get that:
\begin{align*}
    x_1(\theta_{3,2}-h\theta_{3,1})&=x_1\theta_{3,2}-hx_1\theta_{3,1}\\
    &=(x_1-x_2)\theta_{3,2}+x_2\theta_{3,2}-hx_1\theta_{3,1}\\
    &=(g_2\theta_2+g_3\theta_3)+x_2\theta_{3,2}-h(f_2\theta_2+f_3\theta_3)\\
    &=(g_2-hf_2)\theta_2+(g_3-hf_3)\theta_3+x_2\theta_{3,2}\\
    &=x_2(g_2'-hf_2')\theta_2+x_2(g_3'-hf_3')\theta_3+x_2\theta_{3,2}.
\end{align*}
Consequently, $x_1(\theta_{3,2}-h\theta_{3,1} )\in x_2D(\A_{2})$, implying $\theta_{3,2}-h\theta_{3,1}\in x_2D(\A_{1,2})$.
We may assume that:
\begin{align*}
    \theta_{3,2}-h\theta_{3,1}=x_2\varphi , \text{ where } \varphi\in D(\A_{1,2}).
\end{align*}
This leads to:
\begin{align}\label{useless-bu1}
    \theta_{3,2}-x_1\varphi = h\theta_{3,1}-(x_1-x_2)\varphi.
\end{align}

Since the left side belongs to $D(\A_2)$, and the right side belongs to $D(\A_1)$, this implies that $\theta_{3,2}-x_1\varphi\in D(\A)$.
Note that $\theta_{3,2}$ is a level element for $D(\A_2)$. They imply that
$x_1\varphi$ is a level element for $D(\A_2)$. 
By \cref{level_elt_with_ker}, we know that $D(\A_2) = D(\A) + \ker \rho_{2}^1$. 

Now, let's organize the proof on a case-by-case basis.

\begin{enumerate}[label=(\arabic*)]
    \item 
    {
If $D(\A_1) = D(\A) + \ker \rho_{1}^2$, by \cref{min_gen_set} \ref{useless-1}, we have $c_1=c_2$. 
Conversely, we get $h\in \K^*$. Let $h=1$. By \cref{useless-bu1},  this implies that $\theta_{3,1}-(x_1-x_2)\varphi\in D(\A)$.
Note that $\theta_{3,1}$ is a level element for $D(\A_1)$. They imply that
$(x_1-x_2)\varphi$ is a level element for $D(\A_1)$. 
By \cref{level_elt_with_ker}, leading to the deduction that $D(\A_1) = D(\A) + \ker \rho_{1}^2$.

If \(c_1 = c_2 = \max\{d_1, d_2, d_3\}\), then by \cref{free_resolution_Aij} \ref{useless-case2-1} and the definition of an SPOG arrangement, we have that \(\A_{1,2}\) is SPOG with  level \(\max\{d_1, d_2, d_3\}\). Furthermore, if \(\A_{1,2}\) is SPOG, by employing \cref{comfirm_conj} and \cref{diml}, it follows that \(c_1 = c_2 = \max\{d_1, d_2, d_3\}\). }

     \item 
Given that $c_1<c_2$, by \cref{min_gen_set} \ref{useless-1}, we establish that $D(\A_1) \subsetneq D(\A) + \ker \rho_{1}^2$. 

If $D(\A_1) \subsetneq D(\A) + \ker \rho_{1}^2$ and $D(\A_2) = D(\A) + \ker \rho_{2}^1$ with $c_1=c_2$, referring to the previous case, we encounter a contradiction.
\end{enumerate}
\end{proof}

Now, we are going to prove \cref{free_resolution_in_dim3}.

\begin{proof}
Let $| \A_{H_1\cap H_2}| >2$.
Recall $S=\K [x_1,x_2,x_3]$ and $\EP{\A}=(1,d_2,d_3)$. 
We may assume that  $S'=\K [x_1,x_3]$.
Since $\A_{1} ^H=\A_2^H=\A^H=\C$, we may let the Ziegler restriction of $\A, \A_1$ and $\A_{2} $ on $H$ be $(\C,m_0),(\C,m_1)$ and $(\C,m_2)$, respectively. Note that $H_1\cap H=H_2\cap H$. Thus  $(\C,m_1)=(\C,m_2)$.

Since $\A$ is free, by  \cref{Ziegler free}, we may get that $(\C,m_0)$ is free with $\EP{\C,m_0}=(d_2,d_3)$.  
Note that we do not differentiate between $d_2$ and $d_3$.
Since
 $|m_0| =|m_1|+1$ and $m_0(L)\geq m_1(L)$ for any $L\in \C$, by \cref{multi -1}, we may assume that $(\C,m_1)$ is free with $\EP{\C,m_1}=(d_2-1,d_3)$.
Moreover, there exist a basis $\varphi_1,\varphi_2$ for $D(\C,m_1)$ such that $x_1\varphi_1,\varphi_2$ forms a basis for $D(\C,m_0)$. Let us look at the following exact sequences:
\begin{center}
   \begin{tikzcd}
0 \arrow[r] & D_H(\A) \arrow[r, "\cdot x_2"] & D_H(\A) \arrow[r, "\pi"] & D(\C,m_0),
\end{tikzcd}

  \begin{tikzcd}
0 \arrow[r] & D_H(\A_i) \arrow[r, "\cdot x_2"] & D_H(\A_i) \arrow[r, "\pi_i"] & D(\C,m_1),
\end{tikzcd}
\end{center}
where $i=1,2$.

 By  \cref{Ziegler free}, there exist a basis $\theta_1,\theta_2$ for $D_H(\A)$ such that
 $\pi (\theta_1)=x_1\varphi_1$ and $\pi (\theta_2)=\varphi_2$.

Using  \cref{coff-of-level-elt}, we may assume that
$\theta_{3,i}$ is a level element  for $D(\A_i)$ such that
\begin{align}
  x_1\theta_{3,1}&=f_1\theta_1+f_2\theta_2,f_1,f_2\in \K [x_2,x_3]. \label{h1t1 in b}\\
  (x_1-x_2)\theta_{3,2}&=g_1\theta_1+g_2\theta_2,g_1,g_2\in \K [x_2,x_3].  \label{h2t2 in b}
\end{align}
Thus 
\begin{align*}
\pi(x_1\theta_{3,1})&=\overline{f_1}x_1\varphi_1+\overline{f_2}
\varphi_2=x_1\pi_1 (\theta_{3,1}) \\
\pi((x_1-x_2)\theta_{3,2})&=\overline{g_1}x_1\varphi+\overline{g_2}\varphi_2=x_1\pi_2 (\theta_{3,2})  
\end{align*}
Note that $\pi_i (\theta_{3,i})\in D(\C,m_1)=S'\varphi_1+ S'\varphi_2$. 
Thus $\overline{f_2},\overline{g_2}\in x_1S'$.
Recall that
$ f_j,g_j\in \K [x_2,x_3]$, where $j=1,2$,  thus $\overline{f_2}=\overline{g_2}=0$, i.e., $f_2,g_2\in x_2\K[x_2,x_3]$.

By the definition and $\theta_{3,1}\in D_H(\A_1)$, we have $\pi_1(\theta_{3,1})=\rho_1^2(\theta_{3,1})$.
 By \cref{level_elt_with_ker}, we have $\rho_1^2(\theta_{3,1})\neq 0$. This implies that $\overline{f_1}\neq 0$. Since $f_1\in \K[x_2,x_3]$, we may get that {$f_1\in \K[x_2,x_3]\setminus x_2\K[x_2,x_3]$. 
 Similarly, we have $g_1\in \K[x_2,x_3]\setminus x_2\K[x_2,x_3]$.} 
Thus we may assume that
\begin{align*}
    &f_1=x_2f_1'+kx_3^{r},f_2=x_2f_2'.\\
    &g_1=x_2g_1'+k'x_3^{r'},g_2=x_2g_2'.
\end{align*}
where $f_i',g_i'\in \K [x_2,x_3]$, $r,r'\in \Z^+$ and $k,k'\in \K ^*$.
{Note that we assumed $| \A_{H_1 \cap H_2}| > 2$. By \cref{part_pf_thm_in_dim3}, $c_1 = c_2$ if and only if $D(\A_1) = D(\A) + \ker \rho_{1}^2$. By \cref{part_pf_thm_in_dim3}, if $c_1 < c_2$, then $D(\A_1) \subsetneq D(\A) + \ker \rho_{1}^2$ and $D(\A_2) = D(\A) + \ker \rho_{2}^1$. In this case, if $|\DS{\A_{1,2}}| = \ell + 1$, by \cref{min_gen_set} \ref{useless-2}, it is evident that $c_1 < c_2 = d_2$ or $d_3$. By \cref{diml}, we have $\max\{d_2, d_3\} \leq c_1$, which is a contradiction.
Thus, if $c_1<c_2$ we have $|\DS{\A_{1,2}}| = \ell + 2$. 

As a conclusion,  $| \A_{H_1 \cap H_2}| > 2$ if and only if either $D(\A_1) = D(\A) + \ker \rho_{1}^2$ and $c_1 = c_2$, or $|\DS{\A_{1,2}}| = \ell + 2$ and $c_1 < c_2$. 

By \cref{min_gen_set}, we have $|\DS{\A_{1,2}}| \leq \ell + 2$ when $\ell = 3$. Combining with \cref{free_resolution_Aij}, we successfully conclude the proof of the theorem. The moreover part in \cref{free_resolution_in_dim3} \ref{useless-thm1p3} is from \cref{part_pf_thm_in_dim3} \ref{useless-lem4p2}.}
\end{proof}

\begin{remark} 
As demonstrated in \cref{free_resolution_in_dim3}  \ref{useless-thm1p3}, it is evident that $\A_{1,2}$ is not necessarily SPOG.
\end{remark}

Now we would like to provide some examples to apply  \cref{free_resolution_in_dim3}.

\afterpage{
\begin{figure}[h!]
 
\begin{center}
    
\begin{tikzpicture}
  \draw [blue](0,-3) -- (0,3) node[above] {$H_2: x_2=0$};
    \draw (3,0) -- (-3,0) node[left]  {$H_1: x_1=0$};
  \draw (5,-2) arc (0:90:7);
   \draw (6,-2) node[below] {$H_3: x_3=0$};
   \draw [blue](3,1) -- (-3,1) node[left] {$H_5: x_1-x_3=0$};
  \draw (1,2) -- (1,-4) node[below] {$H_7:x_2-x_3=0$};
   \draw (2.5,2.5) -- (-2,-2) node[below] {$H_6: x_2-x_1=0$};
    \draw (3.5,-2.5) -- (-2,3) node[left] {$H_4: x_2+x_1-x_3=0$};
\end{tikzpicture}
\end{center}

\caption{Free arrangement $\mathcal{A}$ in $\mathbb{P}^2$}
\label{fig:pic1}
\end{figure}

\begin{figure}[H]
    \centering
\begin{tikzpicture}
  \draw [blue]  (4,0) -- (-4,0)  node[left]  {$H_2:x_1=0$};
  \draw  [orange] (0,-3) -- (0,4)node[above] {\qquad $H_1:x_3=0$};
  \draw  [purple](6,-2) arc (0:90:7);
   \draw  [purple](5,3) node[below] {\qquad $H_3:x_2=0$};
    \draw  [purple](2.5,-2.5) -- (-4,4) node[above] {$H_4:x_3+x_1=0$};
      \draw[blue] (4,-1) -- (-4,-1) node[left] {$H_5:x_1+x_2=0$};
   \draw [orange](4,1) -- (-4,1) node[left] {$H_6:x_1-x_2=0$};
  \draw [blue] (-1,4) -- (-1,-4) node[below] {$H_7:x_3+x_2=0$};
   \draw[blue] (3,3) -- (-3,-3) node[below] {$H_8:x_3-x_1=0$};
    \draw[blue] (-4,3) -- (3,-4) node[below] {\qquad \qquad $H_9:x_3+x_1+x_2=0$};
     \draw  [orange](-2,4) -- (4.5,-2.5) node[below] {\qquad \qquad $H_{10}:x_3+x_1-2x_2=0$};
      \draw[blue] (1.5,-3.5) -- (-4,2) node[left] {$H_{11}:x_3+x_1+2x_2=0$};
\end{tikzpicture}
\caption{Free arrangement $\mathcal{A}$ in $\mathbb{P}^2$}
\label{fig:pic2}
\end{figure}

\clearpage
}

\begin{example}\label{ex_pic1}
Let $$Q(\A)=x_1x_2x_3(x_1-x_2)(x_1-x_3)(x_2-x_3)(x_1+x_2-x_3).$$
Then, $\A$ is free with $\EP{\A}=(1,3,3)$. See  \cref{fig:pic1}.

By \cref{diml}, it follows that  $\A_2$ and $\A_5$ both are plus-one generated with exponents $(1,3,3)$ and level $3$. By \cref{free_resolution_in_dim3}, we have $\A_{2,5} $ is plus-one generated with $\PO{\A_{2,5}}=(1,2,3)$ and level $3$.
\end{example}



\begin{example}\label{not-spog}
Let 
\begin{align*}
  Q(\A)=&x_1x_2x_3(x_1+x_2)(x_1-x_2)(x_1+x_3)(x_1-x_3)(x_2+x_3)\\
  &(x_1+x_2+x_3)(x_1+2x_2+x_3)(x_1-2x_2+x_3).
\end{align*}
Then, $\A$ is free with $\EP{\A}=(1,5,5)$. 
See \cref{fig:pic2}.

By  \cref{diml}, it is easy to see that $\A_1$, $\A_6$ and $\A_{10}$ are free with exponent $(1,4,5)$, $\A_3$ and $\A_4$ are SPOG with $\PO{\A_3}=\PO{\A_4}=(1,5,5)$ and level $6$, and the remaining $\A_j$ are all SPOG with $\PO{\A_j}=(1,5,5)$ and level $5$. By \cref{free_resolution_in_dim3}, this implies the following results:

\begin{enumerate}[label=(\arabic*)]

\item Note that $\A_2$ and $\A_{11}$ are not free, and $| \A_{H_2\cap H_{11}}| =2$, we may get that $\DS{\A_{2,11}}=(1,5,5,5,5)$,  and the minimal free resolution of $D(\mathcal{A}_{2,11})$ has the following forms:
 \begin{align*} 
        0 \rightarrow S[-6]^2 \rightarrow 
       S[-5]^4 \oplus S[-1] \rightarrow D(\A_{2,11}) \rightarrow 0.
    \end{align*}

\item  Note that $\A_3$ and $\A_{4}$ are SPOG with level $6$, and $| \A_{H_3\cap H_{4}}| >2$, we may get that $D(\A_{3,4})$ is not SPOG,  $\DS{\A_{3,4}}=(1,5,5,5)$,
 and the minimal free resolution of $D(\mathcal{A}_{3,4})$ has the following forms:
 \begin{align*} 
        0 \rightarrow S[-7] \rightarrow 
       S[-5]^3 \oplus S[-1] \rightarrow D(\A_{3,4}) \rightarrow 0.
    \end{align*}

\item  Note that $\A_5$ and $\A_{7}$ are SPOG with level $5$, and $| \A_{H_5\cap H_{7}}| >2$, we may get that $\A_{5,7}$ is SPOG with $\PO{\A_{5,7}}=(1,4,5)$ and level $5$. 

\item Note that $\A_3$ is SPOG with level $6$ and $\A_{5}$ is SPOG with level $5$, and $| \A_{H_3\cap H_{5}}| >2$, we may get that $\DS{\A_{3,5}}=(1,5,5,5,5)$, and the minimal free resolution of $D(\mathcal{A}_{3,5})$ has the following forms:
 \begin{align*} 
        0 \rightarrow S[-6]^2 \rightarrow 
       S[-5]^4 \oplus S[-1] \rightarrow D(\A_{3,5}) \rightarrow 0.
    \end{align*}

\end{enumerate}

\end{example}

 \cref{diml} by Abe shows that NT-free-1 arrangements that are non-free are SPOG. 
 It is known that the converse does not hold in general \cite{ah,ys}.
Using our result, we give such an example.
First, recall the following result:

\begin{proposition}[Theorem 6.2 of \cite{spog}]\label{addf}
    Let $\C$ be SPOG
    with a minimal set of
homogeneous generators $\{\gamma_1,\ldots,\gamma_{\ell},\varphi\}$ for $D(\C)$, where $\varphi$ is a level element with a level coefficient $\alpha$.
     Suppose there exists a free addition $\B =\C \cup \{H\}$ of $\C$.
    If $\deg \varphi>\deg \gamma_i$ for any $i\in \{1,2,\ldots,\ell \}$, then
    $H=\ker \alpha$.
\end{proposition}

The following is a SPOG arrangement that does not admit a free addition.

\begin{example}
Let 
$$Q(\A)=x_1x_2x_3(x_1 - x_2)(x_1 + x_2)(x_1 + 2x_2)(2x_1 + x_2)(3x_1 + x_2)(x_2 + x_3)(3x_1 + x_2+x_3),$$
and let $H_1=\{x_1=0\}$ and $H_2=\{x_2=0\}$. 
By computer, $\A$ is free with $\EP{\A}=(1,3,6)$.
We may assume that $\{\theta_1,\theta_2,\theta_3\}$ is a basis for $D(\A)$.
By \cref{free_resolution_in_dim3}, $\A_{1,2}$ is SPOG with $\PO{\A_{1,2}}=(d_1=1,d_2=3,c_1-1=5)$ with level $d_3=6$. 
By \cref{min_gen_set}, we have a minimal generator set for $D(\A_{1,2})$ as $\{\theta_1,\theta_2,\theta_3,\varphi\}$, where $\alpha_2\varphi$ is a level element for $D(\A_1)$. 
Note that $\theta_3$ is the level element of $\A_{1,2}$ with the level coefficient $\alpha=x_2+x_3$. 
Since $\ker \alpha \in \A_{1,2}$, by \cref{addf}, it follows that $\A_{1,2}$ does not admit free addition.
\end{example}

\begin{example}\label{ex:DS_geom}
We give two arrangements 
$\B$ and $\C$ such that $L(\B)\cong L(\C)$
but $DS(\B) \neq DS(\C)$.
Their defining polynomials are
\begin{align*}
Q_{\B}&= 
 x_{1}  x_{2}  x_{3}  (x_{2} - 3 x_{3})  (x_{2} + 3 x_{3})  (x_{1} - x_{3})  (x_{1} + x_{3})  (x_{1} + x_{2})  (x_{1} + x_{2} - 3 x_{3})  (x_{1} + x_{2} + 3 x_{3})\\
Q_{\C}&=  x_{1}  x_{2}  x_{3}  (x_{2} - 3 x_{3})  (x_{2} + 3 x_{3})   (x_{1} - x_{3})  (x_{1} + x_{3})  (x_{1} + x_{2})  (x_{1} + x_{2} - 4 x_{3})  (x_{1} + x_{2} + 3 x_{3}).
\end{align*}
Their derivation degree sequences are
\begin{align*}
\DS{\B}&= (1,5,6,6)\\
\DS{\C}&= (1, 6, 6, 6, 6, 6).
\end{align*}
This example shows the intricate nature of the derivation degree sequence that depends not only on the combinatorics but also on the geometry of the arrangement.
\end{example}


\end{document}